\documentclass[11pt]{amsart}
\usepackage{geometry}    

\usepackage{fullpage}
\usepackage{graphicx}
\usepackage{algorithm}
\usepackage{algorithmic}
\usepackage{booktabs}
\usepackage{hyperref}
\usepackage{enumitem}
\usepackage{wrapfig}
\usepackage{pifont}
\usepackage{amsmath, amsthm, amssymb, amsfonts, mathtools}

\usepackage[english]{babel}

\graphicspath{{Figures/}}

\usepackage[dvipsnames]{xcolor}

\def\R{\mathbf{R}}

\def\1{\mathbf{1}}

\newtheorem{theorem}{Theorem}
\newtheorem{lemma}{Lemma}
\newtheorem{corollary}{Corollary}
\newtheorem{assumption}{Assumption}
\newtheorem{proposition}{Proposition}

\newtheorem{definition}{Definition}

\usepackage{graphicx}
\graphicspath{{Figures/}}

\title[]{Accelerated Convex Optimization with Stochastic Gradients: Generalizing the Strong-Growth Condition}
\author[]{V\' ictor Valls$^{*,\dagger}$, Shiqiang Wang$^+$, Yuang Jiang$^*$, Leandros Tassiulas$^*$\\
$^*$Yale University, $^\dagger$Trinity College Dublin, $^+$IBM T.\ J.\ Watson Research Center}
\thanks{This work has received funding from the European Union’s Horizon 2020 research and innovation programme under the Marie Skłodowska-Curie grant agreement No.\ 795244.}
\begin{document}

\maketitle

\begin{abstract}
This paper presents a sufficient condition for stochastic gradients not to slow down the convergence of Nesterov's accelerated gradient method. The new condition has the strong-growth condition by Schmidt \& Roux as a special case, and it also allows us to (i) model problems with constraints and (ii) design new types of oracles (e.g., oracles for finite-sum problems such as SAGA). Our results are obtained by revisiting Nesterov's accelerated algorithm and are useful for designing stochastic oracles without changing the underlying first-order method. 
\end{abstract}

\section{Introduction}
\label{sec:introduction}

First-order methods (FOMs) for \emph{smooth} convex optimization are central to machine learning and other areas such as control \cite{BNO+14}, networking \cite{PC06}, and signal processing \cite{CBS14}. In short, these methods minimize a convex function  by querying a first-order oracle (i.e., a subroutine that returns a gradient), and their performance is typically measured in terms of oracle calls.  For instance, the classic \emph{gradient method} (GM) obtains an $\epsilon$-approximate solution with $O(1/k)$ gradient computations \cite{Nes04}, whereas Nesterov's \emph{accelerated gradient method}  \cite{Nes83} has a much faster rate of $O(1/k^2)$. The fewer the gradient computations a FOM needs, the better. However, performance in terms of oracle calls is only half of the picture. The \emph{runtime} of a FOM also depends on how fast gradients are computed, which may be slow in practice. For example, the time to compute a gradient in finite-sum problems scales linearly with the dataset size.

Over the last years, several accelerated FOMs have been proposed to reduce the training time of machine learning problems \cite{All17, ZSC18,DES20,LKQ+20,QRZ20}. The approaches replace an \emph{exact} oracle with an \emph{inexact/stochastic} oracle that can estimate gradients faster without degrading the algorithm's convergence rate. For instance, in finite-sum problems, a stochastic oracle can estimate a gradient by selecting a subset of samples uniformly at random   (i.e., a ``mini-batch'') \cite{Bot12,All17}. 
Those methods obtain better performance than their counterparts with exact oracles. However, the algorithms are often tailored to specific settings and, as a result, difficult to extend to solve other problems. Similarly, it is unclear how solutions for existing problems (e.g., for finite-sums and distributed learning) are related despite doing, in essence, the same: replacing an exact oracle with a stochastic oracle. Hence, the question: What conditions must a stochastic oracle satisfy not to slow down the convergence rate of an accelerated method? 

The question above was first addressed by \cite{VBS19}, which showed that the \emph{strong-growth} condition developed by Schmidt \& Roux in \cite{SR13} for stochastic gradient descent (SGD) also works for accelerated methods.\footnote{The condition is known as ``strong-growth'' as it requires that the magnitude of the stochastic gradients is proportional to the magnitude of the exact gradients. } 
However, the {strong-growth} condition is restricted to unconstrained optimization problems and it does not explain variance reduced oracles for finite-sum problems (e.g., SAGA \cite{DBL14} and SVRG \cite{JZ13}), which have attracted the interest of the machine learning community over the last years. As a result, tailored accelerated algorithms have been developed  that couple the first-order method with the subroutine that estimates gradients (e.g., \cite{All17,ZSC18}). 

This paper presents a new sufficient condition for stochastic gradients not to slow down the convergence of Nesterov's accelerated algorithm. The new condition has the strong-growth condition by Schmidt \& Roux \cite{SR13} as a special case, but it allows us to (i) tackle problems with convex constraints and (ii) design variance reduced oracles for problems that where not possible before. For example, finite-sum problems with oracles such as SAGA \cite{DBL14}. 
Our results also bring to the fore algorithm features known by the theoretical optimization community but rarely exploited in machine learning.
In particular, that the accelerated algorithm is the same independently of whether the objective is smooth or smooth and strongly convex, and that when the cost function is strongly convex, it is possible to use a smaller condition number during the first iterations. Hence, obtain a free speed. Our technical results are connected to the works in \cite{Dev13,CDO18}, which study the convergence of accelerated methods with inexact oracles. However, instead of slowing down the learning process to ensure convergence when gradients are inexact, we study the properties that gradients should satisfy to retain the fast convergence rates. 

The rest of the paper is structured as follows. Sec.\ \ref{sec:contributions} presents the main contributions and Sec.\ \ref{sec:unified_fom} the derivation of the technical results. In Sec.\ \ref{sec:applications}, we show how the new variance condition is related to existing applications, which is useful for designing new oracles. Finally, we illustrate the numerical behavior/performance of the new algorithm in Sec.\ \ref{sec:experiments}.


\section{Main Results} 
\label{sec:contributions}

The main contribution of the paper is Theorem \ref{th:main_theorem}, which establishes the convergence of Algorithm~\ref{al:smooth} for minimizing an $L$-smooth and ($\mu$-strongly) convex function with a stochastic oracle. In particular, Algorithm~\ref{al:smooth} obtains optimal convergence rates when the gradient estimates are unbiased and their variance is proportional to the \emph{difference} of the exact gradients. 

\begin{theorem}
\label{th:main_theorem}
Let $C$ be a convex set and $f : C \to \mathbf R$ a ($\mu$-strongly) convex and $L$-smooth function w.r.t.\ $\| \cdot \|$, where $\| \cdot \|$ is the $\ell_2$ norm. Suppose that $L > \mu$ and let $\phi : C \to \mathbf R$ be a non-negative and $\sigma$-strongly convex function. If the estimated gradients $\widetilde \nabla f(x_k)$ satisfy $\mathbf E [\widetilde \nabla f(x_k)] = \nabla f(x_k)$ and
 \begin{align} 
 \mathbf E [ \| \widetilde \nabla f(x_k) - \nabla f(x_k) \|^2 ]  \le \frac{1}{4\lambda } \cdot \frac{A_{k-1}}{A_k}   \| \nabla f(x_k) - \nabla f(x_{k-1}) \|^2,
\label{eq:general_condition_intro}
\end{align}
for all $k\ge1$, Algorithm \ref{al:smooth} with $\lambda \in (0,1]$ ensures 
\begin{align}
f(y_k) - f(y^\star) \le \frac{\phi(y^\star)}{A_k} 
\label{eq:main_convergence_bound}
\end{align}
where $y^\star \in \arg \min_{u \in C} f(u)$ and 
\begin{align}
A_k \ge \frac{\lambda \sigma}{2L}\prod_{i=1}^k \left(1 + \max \left\{ \frac{2}{i}, \sqrt \frac{\lambda \mu}{L} \right\} \right) - \frac{\lambda \sigma}{2L}.
\label{eq:A_k_lowerbound}
\end{align}
\end{theorem}

In brief, Algorithm \ref{al:smooth} minimizes $f$ by constructing {three sequences}. A sequence $\{x_i \in C\}_{i=1}^k$ of search points to query the first-order oracle, a sequence $\{y_i \in C\}_{i=1}^k$ of approximate solutions, and a sequence of weights $\{ \alpha_i \ge 0  \}_{i=1}^k$. Also, note that the (i) algorithm, (ii) convergence bound, and (iii) variance condition in Eq.\ \eqref{eq:general_condition_intro} are the same regardless of whether the objective function is smooth or smooth and strongly convex. Next, we discuss the convergence properties of Algorithm \ref{al:smooth} and the variance condition in Eq.\ \eqref{eq:general_condition_intro}.

\begin{algorithm}[t]
\begin{algorithmic}[1]
\STATE {\bfseries Input:} ($\mu$-strongly) convex and $L$-smooth function $f$; convex set $C$; non-negative and $\sigma$-strongly convex prox-function $\phi$; stopping criterion.
\STATE \textbf{Set:} $v_0 = \arg \min_{u \in C} \phi(u)$; $A_0 = 0$; $y_0 \in C$; $s_0 = 0$;  $k = 0$,  $\lambda \in (0,1]$
\WHILE {stopping criterion is not met}
\STATE \phantom{$(\bullet)$} \hspace{0.15em} $k \leftarrow k + 1$
\STATE {$(\circ)$} \hspace{0.15em}  Select $\alpha_k$ that satisfies $L {\alpha_k^2}A_k^{-1} = \lambda( \mu A_{k}  + \sigma  )$ 
\STATE  \phantom{$(\bullet)$} \hspace{0.15em}  $A_k \leftarrow A_{k-1}  + \alpha_k$
\STATE {$(\diamond)$} \hspace{0.15em} Select $x_k$ as indicated in Eq.\ \eqref{eq:x_select}
\STATE \phantom{$(\bullet)$} \hspace{0.15em} $\widetilde \nabla f(x_k) \leftarrow$ Query first-order oracle at $x_k$ 
\STATE \phantom{$(\bullet)$} \hspace{0.15em} $s_k \leftarrow s_{k-1} - \alpha_k \widetilde \nabla f(x_k)$
\STATE \phantom{$(\bullet)$} \hspace{0.15em} Select $v_k \in \arg\max_{u \in C} \{ \langle s_k, u \rangle - \phi(u) - \frac{\mu}{2} \sum_{i=1}^k {\alpha_i}\| x_i - u \|^2 \} $ \label{eq:proj_vk} 
\STATE {$(\triangleright)$} \hspace{0.15em} Select $y_k \in Y_k$, where $Y_k$ is the convex set defined in Eq. \eqref{eq:setYk}
\ENDWHILE 
\STATE \textbf{return} $y_k$
\caption{}
\label{al:smooth}
\end{algorithmic}
\end{algorithm}

\textbf{Convergence properties and algorithm complexity.} Algorithm \ref{al:smooth} ensures that $f(y_k) - f^\star \le \phi(y^\star) / A_k$, where  $\phi(y^\star)$ is a constant. 
Thus, the algorithm's convergence rate depends on how fast $A_k$ grows---see Eq.\ \eqref{eq:A_k_lowerbound}. When $f$ is just smooth, we have $f(y_k) - f^\star \le O(1/k^2)$ since $\mu = 0$ and $\smash{\prod_{i=1}^k (1 + 2/i) \sim k^2}$. When $f$ is also $\mu$-strongly convex, the algorithm has  $O(1/k^2)$ convergence when $k \le 2 \sqrt {L / \mu }$, and then it ``switches'' to linear rate since $\smash{1/\prod_{i=1}^k (1+ \sqrt{\mu/L} )\sim \exp(-k \sqrt{\mu/L})}$.\footnote{Note that $(1 + a)^{-k} = \exp(-k \log(1+a)) \le \exp(-k a \log(2)) = 2 \exp(-k a)$ for $a \in [0,1]$.} This behavior is due to our  choice of weights $\alpha_k$  instead of sequences that grow at a predetermined rate (e.g., \cite{CDO18,LKQ+20,BBFG21}). The latter is important in practice as a $1/k^2$ convergence is faster during the first phase of the algorithm. Or equivalently, our algorithm can use a smaller ``condition number'' during the first $k \le 2 \sqrt {L /\mu}$ iterations. 
Nesterov discovered such property for unconstrained optimization problems (see, for example, \cite[Sec.\ 2.2.\ and pp.\ 76]{Nes04}), but it has surprisingly not been exploited by machine learning algorithms despite providing a free speed up. We show the importance of this feature in the numerical experiment in Sec.\ \ref{sec:least_squares}.

 The main steps of the algorithm are $(\circ)$, $(\diamond)$, and $(\triangleright)$, all of which can be computed in closed-form (see also discussion in Sec.\ \ref{sec:selecting_sequences}). Step 10 corresponds to a projection or proximal step and requires solving an auxiliary convex problem (the complexity of which depends on the problem structure).\footnote{When $C = \mathbb R^n$ and $\| \cdot \|$ is the $\ell_2$-norm, Step 10 can be computed in closed-form.}

\textbf{Variance condition and practical implications.} The key feature of Algorithm \ref{al:smooth} is that it allows gradients to be stochastic when those are unbiased and their variance satisfies the condition in Eq.\ \eqref{eq:general_condition_intro}. And when the problem is unconstrained, we show in Corollary \ref{th:prop_unconstrained} (Sec.\ \ref{sec:bounding_nk}) that we can obtain the strong-growth condition by Schmidt \& Roux \cite{SR13} as a special case. 

Parameter $\lambda$ allows us to control the robustness of Algorithm \ref{al:smooth} to stochastic gradients. 
In particular, by making $\lambda$ smaller, the algorithm can handle gradient estimates with larger variance. However, higher robustness is not for free.  Observe from Eq.\ \eqref{eq:A_k_lowerbound} that $\lambda$ scales the strong convexity constants of the problem (i.e., $\sigma$ and $\mu$), and so by making $\lambda$ smaller, we are indirectly ``shrinking'' those constants. 
Such insight is crucial. While parameter $\lambda$ does not affect the algorithms' convergence rate, it does affect the constant factors and so  the \emph{total} number of iterations required to obtain an $\epsilon$-approximate solution. Hence, if we use a stochastic oracle to estimate gradients fast, that extra speed must compensate the additional iterations to, overall, reduce the algorithm runtime (wall-clock). We illustrate the latter in Appendix \ref{sec:app_federated} with a federated optimization problem, where stochastic gradients are used to reduce the communication overhead. 

Finally, we show in Sec.\ \ref{sec:finitesums} how Eq.\ \eqref{eq:general_condition_intro} captures the essence of gradient sampling techniques in finite-sum problems (e.g., SAGA) that cannot be explained with the strong-growth condition \cite{SR13}. The result is important not only because we can encompass existing techniques for estimating gradients, but also because it gives us a new condition that we can exploit to design new types of oracles. For example, to reduce the number of bits transmitted in federated optimization, which is currently only done with schemes that use the strong-growth condition.

\textbf{Technical novelty.} Algorithm \ref{al:smooth} is based on a novel combination of Nesterov's acceleration \cite{Nes05} and dual-averaging techniques for non-smooth optimization \cite{Nes09}, and it encompasses other algorithms (i.e., \cite{Nes05,CDO18,JKM20}) as special cases (see Sec.\ \ref{sec:selecting_sequences}). In particular, our $x_k$ update is a generalization of Nesterov's momentum update $x_k = (A_{k-1} /  A_k ) y_{k-1} + (\alpha_k / A_k) v_{k-1}$ for strongly convex objectives, which is crucial for obtaining the variance condition in Eq.\ \eqref{eq:general_condition_intro} (see also Corollary \ref{th:prop_unconstrained} in Sec.\ \ref{sec:bounding_nk}). The $y_k$ update in Algorithm \ref{al:smooth} also unifies previous algorithms when $f$ is just smooth. When 
\[
y_k \in \arg \min_{u \in C} \left\{  \langle \widetilde \nabla f(x_k),   u -  \hat y_k \rangle +  \frac{L}{2} \| u - x_k \|^2 \right\},
\]
 where $\hat y_k \coloneqq \frac{A_{k-1}}{A_k} y_{k-1} + \frac{\alpha_k}{A_k} v_{k}$, we recover the algorithm in \cite{Nes05}, and if $y_k = \hat y_k$, the algorithms in \cite{GY16,CDO18,JKM20}. Our approach is more flexible in terms of generating approximate solutions, however, it is not clear how we can exploit such flexibility to improve the algorithm's performance. Furthermore, selecting $y_k = \hat y_k$ is often the simplest option from a computational perspective as we can obtain $\hat y_k$ in closed form. Finally, our choice of step size follows the rationale used by Nesterov in \cite[pp.~71]{Nes04}, but we adapt it to the dual-averaging approach and include a parameter $\lambda$ to control the algorithm's robustness.
 

\section{Technical Approach}
\label{sec:unified_fom}

This section contains the derivation of the paper's main technical results. We present the results using a constructive approach that does not presume the steps of our algorithm. All the proofs are in the supplementary material. 

\subsection{Preliminaries}
\label{sec:preliminaries}

We want to solve the optimization problem:
\[
\underset{y \in C}{\text{minimize}} \quad f(y)
\]
where $f : C \to \R$ and $C \subseteq \R^n$ are convex. Let $y^\star \in \arg \min_{u \in C} f(u)$, $f^\star:= f(y^\star)$, and $\| \cdot \|$ denote any $\ell_p$ norm with $p \ge 1$. The dual norm  is defined in the standard way: $\| \cdot \|_* \coloneqq \sup_{u} \{ \langle u, \cdot \rangle \mid \| u \| \le 1\} $. 
We recall the following two definitions.

\begin{definition}[Smoothness]
\label{def:smoothness}
A function $f$ is $L$-smooth w.r.t.\ $\| \cdot \|$ if for all $x,y \in C$
\[f(y) \le f(x) + \langle \nabla f(x), y-x \rangle + \frac{L}{2} \| y - x \|^2.
\] 
\end{definition}
\begin{definition}[Strong convexity]
\label{def:strong_convexity}
A convex function $f$ is $\mu$-strongly convex w.r.t.\ $\| \cdot \|$ if for all $x,y \in C$
\[
f(y) \ge f(x) + \langle \nabla f(x), y-x \rangle + \frac{\mu}{2} \| y - x \|^2.
\] 
\end{definition}

\subsection{Setup and approach overview}

\underline{\textbf{Oracle.}} We minimize $f$ by querying an inexact first-order oracle $\smash{\widetilde {\mathcal O}}$. The gradients returned by the oracle $\smash{\widetilde {\mathcal O}}$ have the form
\begin{align}
\widetilde \nabla f(x_k) \coloneqq \nabla f(x_k) + \xi_k ,
\label{eq:noisemodel}
\end{align}
where $\xi_k \in \R^n$ is an additive perturbation. This inexact oracle model is very general and encompasses other models (e.g., deterministic, stochastic, adversarial) as special cases---see \cite{Asp08, Dev13} for other types of oracle models. Also, we will assume in this section that $\xi_k$ is an arbitrary vector and study the case where this is a random vector in Sec.\ \ref{sec:bounding_nk}. The case where gradients are exact corresponds to having $\xi_k = 0$ for all $k \ge 1$. 

\underline{\textbf{Sequences.}} 
As mentioned in Sec.\ \ref{sec:contributions}, we minimize $f$ by constructing {three sequences}. A sequence $\{x_i \in C\}_{i=1}^k$ of search points (i.e., to query the first-order oracle), a sequence $\{y_i \in C\}_{i=1}^k$ of approximate solutions, and a sequence of weights $\{ \alpha_i \ge 0  \}_{i=1}^k$. Besides the three main sequences, we also define the auxiliary sequences:\footnote{These sequences are typical in dual-averaging approaches; see, for example, \cite[Eqs.\ (2.4) and (2.9)]{Nes09}.}
\begin{align} 
 A_k \coloneqq \sum_{i=1}^k \alpha_i, 
 \qquad  s_k \coloneqq - \sum_{i=1}^k \alpha_i \widetilde \nabla f(x_i), 
 \qquad  v_k \in \arg\max_{u \in C} \left\{ \langle s_k, u \rangle - \varphi_k(u) \right\} \label{eq:v_select}
\end{align}
where 
\[ \varphi_k(u) \coloneqq  \phi(u) + \frac{\mu}{2} \sum_{i=1}^k {\alpha_i}\| x_i - u \|^2
\]
and $\phi$ is a $\sigma$-strongly convex function w.r.t.\ $\| \cdot \|$.
In brief, $A_k$ is the sum of the weights, $s_k$ the  negative sum of the weighted first-order information, and $v_k$ the vector in $C$ that maximizes the strongly concave function $\langle s_k, u \rangle - \varphi_k(u)$ w.r.t.\ $u$. Function $\varphi_k$ can be thought as a time-varying strongly convex prox-function when $f$ is $\mu$-strongly convex---otherwise $\varphi_k = \phi$ since $\mu = 0$. Also, and to streamline exposition, we make the following assumption.

\begin{assumption}
\label{as:norm}
$\varphi_k$ is $(\mu A_k + \sigma)$-strongly convex w.r.t. the norm $\| \cdot\|$.
\end{assumption} 
Assumption \ref{as:norm} is trivially satisfied for any $\ell_p$ norm when $f$ is just $L$-smooth. When $f$ is, in addition, $\mu$-strongly convex, the assumption is satisfied by any $\ell_p$ norm that is strongly convex with modulus larger than $1$. We assumed in Theorem \ref{th:main_theorem} that $\| \cdot \|$ is the $\ell_2$ norm to streamline exposition.

\underline{\textbf{Strategy.}} Our method of proof consists of constructing an upper bound on $f(y_k)$ \emph{without specifying where to sample gradients at intermediate points $i=1,\dots,k-1$}. 
The method is in contrast to steepest-descent methods (e.g., the classic gradient algorithm) where (i) the $k$'th approximate solution depends \emph{only} on the information in the $(k-1)$'th iterate, and (ii) approximate solutions are used to query the oracle in the next iteration (i.e., $x_k = y_{k-1}$). The use of two different sequences $\{x_i \in C\}_{i=1}^k$, $\{y_i \in C\}_{i=1}^k$ (i.e., evaluating the objective function and the gradients at different points) is naturally motivated by Definitions \ref{def:smoothness} and \ref{def:strong_convexity}. We have the following lemma.

\begin{lemma}
\label{th:dual_averaging_bound}
Let $f : C \to \R^n$ be a ($\mu$-strongly) convex function with $L$-Lipschitz continuous gradient. Select a $y_0 
\in C$ and let $A_0 = 0$. Suppose Assumption \ref{as:norm} holds and let $\phi$ be a non-negative and $\sigma$-strongly convex. For any sequences $\{x_i \in C\}_{i=1}^k$, $\{y_i \in C\}_{i=1}^k$,  $\{ \alpha_i \ge 0 \}_{i=1}^k$, we have 
\begin{align} 
 A_k (f(y_k) - f^\star)   & \le \phi(y^\star) + N_k +  \sum_{i=1}^k  \Bigg(  A_i \left\langle \widetilde \nabla f(x_i),   y_i -  \frac{A_{i-1}}{A_{i}} y_{i-1} - \frac{\alpha_i}{A_i} v_i \right\rangle +  \frac{L}{2} A_i \| y_i - x_i \|^2  \notag \\
& \quad 
-  \frac{1}{2L} A_{i-1} \| \nabla f(x_i) - \nabla f( y_{i-1}) \|_*^2   - \frac{\mu A_{i} + \sigma}{2} \left\| v_i - \hat v_{i-1}    \right\|^2 \Bigg)   \label{eq:mirror_bound} 
 \end{align}
where $N_k \coloneqq - \sum_{i=1}^k    \langle \xi_i,  A_i y_i - A_{i-1} y_{i-1} - \alpha_i y^\star  \rangle$ and  $ \hat v_{k-1} = \frac{\mu A_{k-1} + \sigma}{\mu A_{k} + \sigma} v_{k-1} + \frac{\mu \alpha_k}{\mu A_{k} + \sigma} x_k$.
\end{lemma}

Importantly, Eq.\ \eqref{eq:mirror_bound} holds for arbitrary sequences $\{x_i \in C\}_{i=1}^k$, $\{y_i \in C\}_{i=1}^k$,  and $\{ \alpha_i \ge 0 \}_{i=1}^k$.  
In brief, the first term in the right-hand-side of  Eq.\ \eqref{eq:mirror_bound} is \emph{a constant} that depends on an optimal solution $y^\star$ and the choice of prox-function $\phi$.
The second term, $N_k$, is equal to zero if $\xi_k = 0$ for all $k\ge 1$ (i.e., the gradients are exact). The third term is the sum of $k$ terms and controls how the right-hand-side of Eq.\ \eqref{eq:mirror_bound} grows. Observe that if the third-term is non-positive and $N_k \le 0$, we have that $ f(y_k) - f^\star \le {\phi(y^\star)}/{A_k}$, and so the convergence rate of the algorithm depends on how fast $A_k$ grows since $\phi(y^\star)$ is a constant. Such an approach is, in essence, equivalent to the \emph{estimate sequence} technique employed by Nesterov (see, e.g., \cite[Sec.\ 2.2.1]{Nes04}). The next step is to select the three sequences.

\subsection{Selecting sequences $\{ x_i \in C \}_{i=1}^k$, $\{ y_i \in C \}_{i=1}^k$, $\{ \alpha_i \ge 0 \}_{i=1}^k$, and convergence}
\label{sec:selecting_sequences}

With the result in Lemma \ref{th:dual_averaging_bound}, one may want to find the values $x_k \in C$, $y_k \in C$ and $\alpha_k \ge 0$ that \emph{minimize} the right-hand-side of Eq.\ \eqref{eq:mirror_bound} for fixed sequences $\{ x_i \in C \}_{i=1}^{k-1}$, $\{ y_i \in C \}_{i=1}^{k-1}$, $\{ \alpha_i \ge 0 \}_{i=1}^{k-1}$. Unfortunately, finding such values is difficult as $v_k$ depends on the gradients and weights up to time $k$. Namely, $x_k$ and $\alpha_k$ cannot be selected for a fixed $v_k$, as $v_k$ depends on ${\widetilde \nabla f(x_k)}$ and $\alpha_k$. Instead, we select sequences to \emph{control} how the third-term in the right-hand-side of Eq.\ \eqref{eq:mirror_bound} grows. 
The strategy is to make the term
\begin{align}
A_k \left\langle \widetilde \nabla f(x_k),   y_k -  \frac{A_{k-1}}{A_{k}} y_{k-1} - \frac{\alpha_k}{A_k} v_k \right\rangle   +  \frac{L}{2} A_k \| y_k - x_k \|^2 
\label{eq:fund_ineq}
\end{align}
\emph{proportional} to $\| \hat v_k - v_k \|^2$, which will allows us to control the convergence rate of the algorithm with sequence $A_k$. 
We do that as follows. An approximate solution $y_k$ is selected from the convex set
\begin{align}
Y_k \coloneqq  \left\{ u \in C   : \text{Eq.}\ \eqref{eq:fund_ineq} \le \frac{L}{2} \|\hat y_k - x_k\|^2 \right\}, \label{eq:setYk}
\end{align} 
where $
\hat y_k \coloneqq \frac{A_{k-1}}{A_k} y_{k-1} + \frac{\alpha_k}{A_k} v_{k}$. The search point is selected as
\begin{align}
 x_k & = \frac{(\mu A_k +\sigma) A_{k-1}}{\mu (A_{k} - \alpha_k) (A_{k} + \alpha_k) + \sigma A_k } y_{k-1}  + \frac{(\mu A_{k-1} + \sigma) \alpha_k}{\mu (A_{k} - \alpha_k)(A_k + \alpha_k) + \sigma A_k} v_{k-1}, \label{eq:x_select}
\end{align}
which ensures 
\begin{align*}
\text{Eq.}\ \eqref{eq:fund_ineq} \le \smash{\|  \hat y_k - x_k \|^2 \le \frac{\alpha_k^2}{A_k^{2}} \| v_k -   \hat v_{k-1} \|^2}.
\end{align*}
Hence, with a $y_k \in Y_k$ and $x_k$ as in Eq.\ \eqref{eq:x_select}, we can upper bound Eq.\ \eqref{eq:fund_ineq} by 
\begin{align}
\left(L \frac{\alpha_k^2}{A_k} - \mu A_{k} - \sigma \right)   \left\|  v_k - \hat v_k   \right\|^2.
\label{eq:c_bound_acc}
\end{align}
Finding the weights that minimize Eq.\ \eqref{eq:c_bound_acc} is, again, a non-tractable optimization problem as $v_k$ depends on $\alpha_k$. Instead, we maximize $\alpha_k$ subject to ${L {\alpha_k^2}A_k^{-1} \le \mu A_{k}  + \sigma }$ as this makes Eq.\ \eqref{eq:c_bound_acc} non-positive for all $k\ge1$, and so we can drop the term. In particular, we select 
\begin{align}
L \frac{\alpha_k^2}{A_k} = \lambda( \mu A_{k}  + \sigma  )
\label{eq:alpha_select}
\end{align}
 where $\lambda \in (0,1]$ is a parameter that allows us to replace the inequality constraint with an equality.
We can obtain a weight $\alpha_k \ge 0$ that satisfies $L {\alpha_k^2}A_k^{-1} = \lambda( \mu A_{k}  + \sigma  )$ by finding the positive root of the quadratic equation: 
\begin{align}
L \alpha_k^2 - \lambda \mu A_{k}^2  - \lambda \sigma A_{k} = ( L - \mu ) \alpha_k^2   - \lambda (2 \mu A_{k-1} + \sigma)\alpha_k - \lambda (\mu A_{k-1}^2  + \sigma A_{k-1}), \label{eq:alpha_select}
\end{align} which can be obtained in closed form. 

Finally, selecting $y_k$, $x_k$ and $\alpha_k$ as indicated above, we obtain the bound
\begin{align}
A_k (f(y_k) - f^\star) & \le \phi(y^\star) + N_k - \frac{1}{2L}\sum_{i=1}^k A_{i-1} \| \nabla f(x_i) - \nabla f( y_{i-1}) \|_*^2 \notag \\
& \qquad - (1-\lambda) \sum_{i=1}^k (\mu A_i + \sigma) \| v_i - \hat v_{i-1} \|^2 .
\label{eq:accelerated_bound}
\end{align}

The third and fourth terms in the right-hand-side of Eq.\ \eqref{eq:accelerated_bound} are non-positive. Hence, if $N_k \le 0$, the convergence rate depends on how fast $A_k$ grows. We leave in the Appendix (Sec.\ \ref{sec:proof_main_theorem}) the proof of the lower bound on $A_k$.

In the next section, we show how to use the third and forth terms in the right-hand-side of Eq.\ \eqref{eq:accelerated_bound} to cancel $N_k$ when this is not zero (due to having inexact gradients). 
%

\subsection{Bounding the noise term $N_k$}
\label{sec:bounding_nk}
 
Now, we suppose that vectors $\xi_k$ are zero-mean and independent random variables. Thus, $\smash{\mathbf E_{\xi_k} [ \widetilde \nabla f(x_k)] = \nabla f(x_k)}$ for all $k\ge 1$.\footnote{This assumption is necessary to obtain a bound on $N_k \coloneqq - \sum_{i=1}^k    \langle \xi_i,  A_i y_i - A_{i-1} y_{i-1} - \alpha_i y^\star  \rangle$ that does not depend on $y^\star$.}
We proceed to upper bound $N_k$ and then show, in Lemma \ref{th:prop_constrained}, a sufficient condition for a stochastic oracle not to degrade Algorithm's \ref{al:smooth} convergence rate. 
\begin{lemma} 
\label{th:bound_unbiased}
Suppose that the noise vectors $\xi_k$ are independent and have zero mean for all $k$, i.e., $\mathbf E[\xi_k] = 0$. Also, select $\lambda \in (0,1]$ and suppose $y_k = \hat y_k \in Y_k$ in Algorithm \ref{al:smooth}. Then,
\begin{align}
\mathbf E_{\xi_1,\dots,\xi_k}[N_k] \le \frac{\lambda}{L} \sum_{i=1}^k  A_i \mathbf E_{\xi_i} [ \| \xi_i \|_*^2 ]
\label{eq:Nk}
\end{align}
\end{lemma}
That is, $N_k$ can be upper bounded by the sum of the variances $\mathbf E_{\xi_i}[ \| \xi_i \|^2_* ]$ from $i=1,\dots,k$ weighted by $ \lambda A_i / L$.  Note the bound holds independently of whether $f$ is smooth or smooth and strongly convex. The assumption that $y_k = \hat y_k \in Y_k$ is useful to streamline the analysis and does not change the result in Theorem \ref{th:main_theorem}. With the result in Lemma \ref{th:bound_unbiased}, we can derive the sufficient condition.

\begin{lemma}
\label{th:prop_constrained}
Let $\lambda \in (0,1]$, and suppose ${\mathbf E [ \widetilde \nabla f(x_k) ]  = \nabla f(x_k)}$ and 
 \begin{align*} 
 \mathbf E [ \| \xi_k \|_*^2 ]  \le \frac{1}{4\lambda } \cdot \frac{A_{k-1}}{A_k}   \| \nabla f(x_k) - \nabla f(x_{k-1}) \|_*^2 ,
 \end{align*}
 hold for all $k\ge1$. Then, the last three terms in the right-hand-side of Eq.\ \eqref{eq:accelerated_bound} are non-positive. 
\end{lemma}
That is, if the variance  $\mathbf E [ \| \xi_k \|_*^2 ] =  \mathbf E [ \| \widetilde \nabla f(x_k) - \nabla f(x_k) \|_*^2 ]$ is proportional to the difference of the exact gradients at $x_k$ and $x_{k-1}$, we have that $\mathbf E[ f(y_k) - f^\star] \le \phi(x^\star)/A_k $. Also, note that if the estimated gradients are exact (i.e., $\widetilde \nabla f(x_k) = \nabla f(x_k)$), the  condition in Lemma \ref{th:prop_constrained} is satisfied trivially.
Parameter $\lambda \in (0,1]$  controls the robustness of Algorithm \ref{al:smooth}. By making $\lambda$ small, Algorithm \ref{al:smooth} can handle stochastic gradients with larger variance. However, at the price of ``shrinking'' the strong convexity constants in the problem---see the discussion after Theorem \ref{th:main_theorem}. 

When the problem is unconstrained, the variance condition can be simplified as follows.

\begin{corollary}[Unconstrained optimization]
\label{th:prop_unconstrained}
Let  $C = \R^n$, $\phi(\cdot) = \frac{1}{2} \| \cdot \|^2$ (i.e.,  $\sigma = 1$) and $\| \cdot \|$ be the $\ell_2$ norm. Suppose ${\mathbf E [ \widetilde \nabla f(x_k) ]  = \nabla f(x_k)}$ for all $k\ge 1$, and that 
\begin{align}
\mathbf E [ \| \xi_k \|_2^2 ] \le  \frac{(1-\lambda)}{(1+\lambda)}  \left\|  \nabla f(x_k) \right\|^2_2
\label{eq:condition_unconstrained}
\end{align}
with $\lambda \in (0,1]$. Then, the last three terms in the right-hand-side of Eq.\ \eqref{eq:accelerated_bound} are non-positive. 
\end{corollary}
The condition now is that the variance is proportional to the magnitude of the exact gradients, which is analogous to the strong-growth condition \cite{SR13}. In short, since $\xi_k = \widetilde \nabla f(x_k) - \nabla f(x_k) $, we have $\mathbf E [ \| \xi_k \|_2^2 ] =  \mathbf E [ \| \widetilde \nabla f(x_k) \|_2^2] -   \| \nabla f(x_k) \|_2^2  $, and so Eq.\ \eqref{eq:condition_unconstrained} can be rewritten as  $\mathbf E [ \| \widetilde \nabla f(x_k) \|_2^2 ] \le \frac{2}{(1+\lambda)}  \| \nabla f(x_k) \|_2^2$. Hence, by making $\lambda$ smaller, we can increase the robustness to stochastic gradients  up to twice the magnitude of the exact gradient. The 2-factor is because how we defined $\widetilde \nabla f(x_k)$ in Eq.\ \eqref{eq:noisemodel}. If we had instead defined $\widetilde \nabla f(x_k) = \beta \nabla f(x_k) +  \xi_k$ with $\beta > 0$, the condition in Eq.\ \eqref{eq:condition_unconstrained} would be equivalent to $
\mathbf E [ \| \widetilde \nabla f(x_k) \|_2^2 ] \le  \frac{2\beta^2}{(1+\lambda)} \| \nabla f(x_k)\|_2^2$, which is exactly the strong-growth condition. See, for example, \cite[Eq.\ (1)]{VBS19}, where parameter $\rho$ in \cite{VBS19}  corresponds to ${2\beta^2}/{(1+\lambda)}$.


\section{Applications}
\label{sec:applications}

This section shows how we can map the condition in Lemma \ref{th:prop_constrained} and Corollary \ref{th:prop_unconstrained} to two applications.

\subsection{Finite-sum problems} 
\label{sec:finitesums}

This application consists of minimizing $f(x) =\sum_{l=1}^m f_l(x)$, where $f_l(x)$ is the loss function for data sample $l \in \{1,\dots,m\}$. Each $f_l(x)$ is ($\mu$-strongly) convex and has $\frac{L}{m}$-Lipschitz continuous gradient. When the number of data samples $m$ is large, the computation of a gradient $\nabla f(x) =  \sum_{l=1}^m \nabla f_l(x)$ is expensive as it involves computing $m$ ``data-sample'' gradients. The classic approach to reducing the computation burden is to estimate a gradient using a subset of $b$ samples. That is, ${\widetilde \nabla f(x_k)} = \frac{m}{b} \sum_{j \in J_k} \nabla f_j(x_k)$ where  $b \coloneqq |J_k|$ and  $J_k \subseteq \{1,\dots,m\}$ is a collection on indexes  selected uniformly at random. Such approach ensures that $\mathbf E_{J_k}[ \widetilde \nabla f(x_k) ] = \nabla f(x_k)$, but not that the $\mathbf E[\| \xi_k \|_*^2]$ vanishes with $k$ or that this is proportional to the difference of the gradients. 

\emph{Variance-reduced} oracles (e.g., SAGA \cite{DBL14}, SVRG \cite{JZ13}) add a ``memory'' term  to the mini-batch gradient to allow the variance to vanish. For instance, \textsc{SAGA} estimates gradients as follows. It sets $x_0 \in C$ and $\psi_0^l = x_0$ for all $l \in \{ 1,\dots,m \}$ as initial points, and then, in each iteration $k \ge 1$, it (i) selects a collection of indexes $J_k \subseteq \{1,\dots,n\}$ uniformly at random, (ii) sets $\psi_k^j = x_k$ for all $j \in J_k$, and (iii) computes 
\begin{align}
\widetilde \nabla f(x_k)  & = \frac{m}{b}\sum_{j \in J_k }  \nabla f_j (x_k) - \frac{m}{b} \sum_{j \in J_k } \nabla f_j (\underbrace{\psi_{k-1}^j}_{\approx x_{k-1}})  + \sum_{l=1}^m \nabla f_l(\underbrace{\psi_{k-1}^l}_{\approx x_{k-1}}).
\label{eq:saga_update}
\end{align}
That is, $\widetilde \nabla f(x_k)$ is estimated using ``all'' data samples, but only few of them are computed with the latest point $x_k$. 
Note that the estimated gradient in Eq.\ \eqref{eq:saga_update} is unbiased (i.e., $\mathbf E_{J_k} [ \widetilde \nabla f(x_k) ] = \nabla f(x_k)$), but also that 
\begin{align*}
\| \xi_k \|_*^2 & = \| \widetilde \nabla f(x_k) - \nabla f(x_k) \|_*^2  \\
& = \Bigg\|  
 \Bigg( \underbrace{\sum_{l=1}^m \nabla f_l (x_k)}_{ = \nabla f(x_k)}  -  \underbrace{\sum_{l=1}^m \nabla f_l({\psi_{k-1}^l})}_{\approx \nabla f(x_{k-1})}
 \Bigg)
 -  \frac{m}{b}  \Bigg( \underbrace{\sum_{j \in J_k }  \nabla f_j (x_k)}_{\approx \nabla f(x_{k})} - \underbrace{\sum_{j \in J_k } \nabla f_j ({\psi_{k-1}^j})}_{\approx \nabla f(x_{k-1})}  \Bigg)
\Bigg\|_*^2 
\end{align*}
which is very close to the condition in Lemma \ref{th:prop_constrained}. In particular,  $\| \xi_k \|_*^2$  is the difference between the gradient at iteration $k$ and the copy of the gradients in the previous iterations since $\psi_{k-1}^l \in \{x_{k-1},\dots,x_0\}$. We show in the supplementary material (Sec.\ \ref{sec:sup_app_saga}) that Eq.\ \eqref{eq:saga_update} satisfies the condition in Lemma \ref{th:prop_constrained}, but the same strategy can be applied to similar oracles or types of problems. 
The most important point of this section is to show how the update in Eq.\ \eqref{eq:saga_update}---which is difficult to motivate initially---makes indeed sense when we connect it to the variance condition in Lemma \ref{th:prop_constrained}. We believe that such hindsight can be useful to design new types of stochastic oracles. 
To conclude, we want to recall that while satisfying the conditions in Lemma \ref{th:prop_constrained} is enough for retaining the acceleration rates, it is important to select a $\lambda \in (0,1]$ that does not ``shrink'' too much the strong convexity constants (see the discussion after Theorem \ref{th:main_theorem}).

\subsection{Federated optimization}
\label{sec:federatedlearning}

This application minimizes $f(x) \coloneqq \sum_{l=1}^m f_l(x)$ in a distributed manner, where $f_l : \R^n \to \R$ are loss functions owned by different clients.  In particular, a central server collects the clients' gradients $\nabla f_l(x_k)$ and then carries out the updates indicated in Algorithm \ref{al:smooth}. To reduce the communication time, the clients quantize\footnote{Using a stochastic oracle.} the ``local'' gradients with fewer bits than what it would take to encode an exact gradient (i.e., $32/64$ bits per floating point). 

Previous work proposed algorithms where a stochastic oracle has to satisfy the strong-growth condition $\mathbf E [ \| \xi_k\|_*^2 ] \le \omega \left\|  \nabla f(x_k) \right\|^2$ where $\omega \ge 0$ (e.g., \cite[Def.\ 1]{LKQ+20}, \cite[Def.\ 2]{HHH+20} and \cite[Sec.\ 1.1]{QRZ20}), which is satisfied by multiple gradient compression techniques. For instance, natural compression \cite{HHH+20}, random sparsification, and random dithering \cite{AGL+17}. 
Our result in Lemma \ref{th:prop_constrained} is important since it provides a new condition to design oracles that can quantize gradients. In particular, the condition in Lemma \ref{th:prop_constrained} suggests that quantization can be carried out by taking into account the information sent in the previous iterations, which is a hindsight not available before.
Furthermore, recall that the condition in Lemma \ref{th:prop_constrained} is more general than the strong-growth condition as it allows us to train machine learning models with constraints. 

\begin{figure*}[t!]
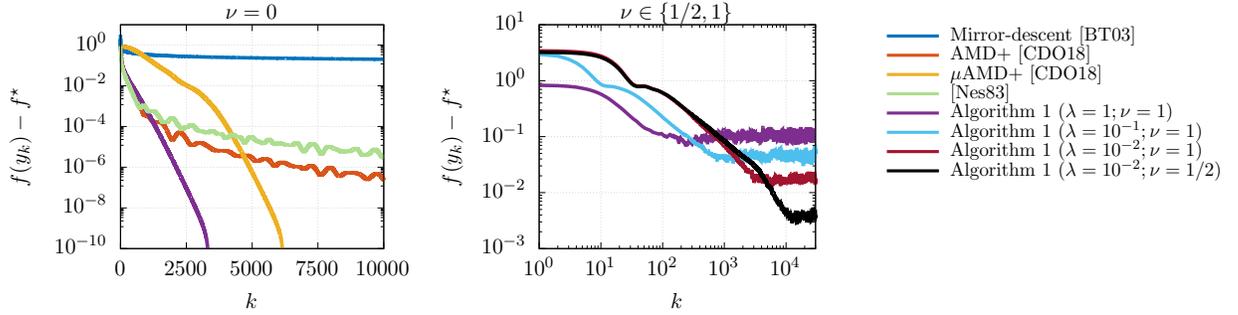

\centering
{\resizebox{0.34\textwidth}{!}{\input{Figures/exp1.tex}}} 
{\resizebox{0.34\textwidth}{!}{\input{Figures/exp1b.tex}}}
{\resizebox{0.2\textwidth}{!}{\input{Figures/key.tex}}} 
\caption{Least squares with noisy gradients: $\widetilde \nabla f(x_k) = \nabla f(x_k) + \xi_k$ where $\xi_k \sim \mathcal N(0,\nu)$ with $\nu \in \{0,1/2,1\}$. The plots are the average of $50$ realizations.}
\label{fig:least_squares}
\end{figure*}

\section{Numerical Experiments}
\label{sec:experiments}

We illustrate Algorithm \ref{al:smooth} with three numerical experiments: (i) a least squares problem where gradients are affected by gaussian noise, (ii) a federated logistic regression problem with compressed gradients (Appendix; Sec.\ \ref{sec:app_federated}), and (iii) a finite-sum problem (Appendix; Sec.\ \ref{sec:app_saga}). 

\subsection{Least squares with noisy gradients}
\label{sec:least_squares}
The goal of this experiment is to show (i) the two-phase behavior of Algorithm \ref{al:smooth} when the objective is strongly convex, and (ii) how parameter $\lambda$ affects the algorithm's convergence/robustness. The cost function is $f(x) = \frac{1}{2}\| Ax- b \|_2^2$ where $A\in \R^{n\times n}$ and $b \in \R^n$ are generated uniformly at random with $n = 50$. 
We compare Algorithm \ref{al:smooth} against Mirror-descent \cite{BT03}, Nesterov's original accelerated algorithm \cite{Nes83}, and \textsc{AMD+}, \textsc{$\mu$AMD+} \cite{CDO18}. Algorithms \cite{Nes83} and \textsc{AMD+} are in the domain of smooth convex optimization, but we include them to illustrate the behavior of Algorithm \ref{al:smooth} during the first iterations.

Figure \ref{fig:least_squares}a (noiseless gradients) shows how Algorithm \ref{al:smooth} ``tracks'' the smooth optimization algorithms \cite{Nes83}, \textsc{AMD+} and then  it ``switches'' to linear rate. This unique behavior is due to the choice of weights in our algorithm (see Eq.\ \eqref{eq:A_k_lowerbound} and the discussion after Theorem \ref{th:main_theorem}). Also, compare Algorithm \ref{al:smooth} with \textsc{$\mu$AMD+}, which has linear convergence but a slow start since it uses the same condition number throughout all iterations. 
Figure \ref{fig:least_squares}b shows the convergence of Algorithm \ref{al:smooth} for different values of $\lambda$ when gradients are stochastic and the noise has unit variance. Observe that by making $\lambda$ smaller (i.e., by ``shrinking'' the condition number), we are increasing the algorithm's robustness. In particular, Algorithm \ref{al:smooth} converges at a linear rate until the noise in the gradients is large enough to affect convergence. Conversely, observe that when the noise variance is smaller ($\nu = 1/2$), Algorithm \ref{al:smooth} converges to a better approximate solution, which is the expected behavior from reducing the variance of the noise in the gradients. 

%


%


\clearpage

\clearpage
\bibliography{references.bib}
\bibliographystyle{alpha}

\clearpage

\clearpage

\appendix

\clearpage
\onecolumn
\section*{Supplementary Material}


\section{Background and related work}
\label{sec:relatedwork}
\textbf{Background.} Accelerated first-order methods (FOMs) for smooth convex optimization were discovered by Nesterov in 1983 \cite{Nes83}, and later extended to strongly convex objectives (e.g., \cite[Sec.\ 2.2]{Nes04}). Despite their superior rates,  
accelerated methods have been historically less employed in practice due to their high sensitivity to inexact gradients and unintuitive method of proof (see \cite{Bub13,SBC14,WWJ16}).\footnote{Compared to descent methods.} The interest in accelerated FOMs methods resurged with large-scale and distributed optimization problems where higher-order methods are impractical despite having better convergence rates.

\textbf{Related work.} There are two lines of work on accelerated first-order optimization with stochastic gradients. The ``black-box'' approaches (e.g., \cite{Asp08,Lin10, GL12,Dev13, CDO18}), which regard an oracle as a  subroutine that cannot be controlled,\footnote{For example, when gradients are estimated with  \emph{gradient-free} techniques \cite[Sec.\ 3]{VGS21} or obtained by solving an auxiliary convex problem \cite{BNO+14}. } and the ``application specific'' approaches (e.g., \cite{All17,DES20,LKQ+20}), which tie the FOMs  with the subroutines that estimate gradients. The underlying mathematical problems are nonetheless the same.

Our work is  close technically to the ``black-box'' approaches  \cite{Lin10, Dev13, CDO18}, which also propose algorithms based on dual-averaging. However, instead of studying how to slow down the learning process to ensure convergence,\footnote{Those works assume that the variance of the stochastic gradients is fixed or uniformly bounded from below.} we study the conditions that the estimated gradients should satisfy to retain the accelerated rates.\footnote{Recall that we can do this as we assume that we have access to the subroutine that estimates gradients, such as in the ``application specific'' approaches.} 
Regarding the ``application specific'' methods, the two main applications are (i) finite-sums with large datasets \cite{All17,ZSC18,ZDS+19,DES20,SJM20} and (ii) distributed learning with compressed gradients \cite{LKQ+20,QRZ20}. 
The works for both applications use oracles proposed for non-accelerated algorithms (e.g., SAGA \cite{DBL14}, SVRG \cite{JZ13}), and show how they can employ them with new acceleration methods.
The work that is conceptually closer to us is \cite{DES20}, which shows how different stochastic oracles for finite-sum problems (i.e., SAGA, SVRG, SARAH \cite{NLS+17}) can be used by a single FOM. However, the condition used \cite[Def.\ 1]{DES20} is tailored to finite-sum problems, and it is not clear how this can be extended to other problems, e.g., distributed learning with compressed gradients. 


\section{Preliminaries}
\label{sec:preliminaries}

This section presents a concise review of some key convex duality results that we will use to prove the technical content. All the results are well-known, and we include the proofs for completeness.

\begin{definition}[Convex conjugate] 
\label{de:convex_conjugate}
\begin{align*}
f^*(x^*) \coloneqq \sup_{x \in C} \{ \langle x^*, x \rangle - f(x) \}
\end{align*}
\end{definition}

The following are some convex conjugate facts. 
\begin{lemma}[Convex conjugate facts] Let $f$ be a convex function on $C$. 
\label{th:conjugate_facts}
\begin{itemize}
\item[(i)] $\nabla f^*(x^*) \in \arg \max_{u \in C} \{ \langle x^*, u \rangle - f(u) \}$ 
\item [(ii)] $f^*(x^*)  =  \langle x^*, \nabla f^*(x^*)   \rangle - f(\nabla f^*(x^*))$ 
\item [(iii)] $x^* = \nabla f(\nabla f^*(x^*))$
\end{itemize}
\end{lemma}
\begin{proof}
For (i), let $u$ be a vector in $C$ that maximizes $\langle x^*, u \rangle - f(u)$. Then, $f^*(x^*) = \langle x^*, u \rangle - f(u)$ (by Definition \ref{de:convex_conjugate}) and therefore $\nabla f^*(x^*) = u $. 
Using (i) in the definition of convex conjugate (Definition \ref{de:convex_conjugate}) yields (ii). 
Finally, (iii) follows by differentiating the right-hand-side of (ii) with respect to $\nabla f^*(x^*) $ and equating it to zero, i.e., $0  = x^* - \nabla f(\nabla f^*(x^*) )$.
\end{proof}

\begin{lemma}[Fenchel's inequality]
\label{th:fenchel_inequality}
$\langle x, x^* \rangle \le f(x) + f^*(x^*)$.
\end{lemma}

\begin{lemma}
\label{th:smooth_strong_duality}
$f$ is $\mu$-strongly convex w.r.t.\ $\| \cdot \|$ $\Longleftrightarrow$ $f^*$ is $1/\mu$-smooth w.r.t.\ $\| \cdot\|_*$.
\end{lemma}
\begin{proof}
The proof uses the results in Lemma \ref{th:conjugate_facts}, Lemma \ref{th:fenchel_inequality}, and well-known duality relations; see, for example, \cite[pp.\ 475]{RW98}. 

$\Longrightarrow$
\begin{align*}
f^*(x^*) 
& = \langle x^*, \nabla f^*(x^*) \rangle  - f (\nabla f^*(x^*)) \\
& \le \langle x^*, \nabla f^*(x^*) \rangle - f (\nabla f^*(y^*)) - \langle \nabla f (\nabla f^*(y^*)), \nabla f^*(x^*) - \nabla f^*(y^*) \rangle \\
 & \qquad - \frac{\mu}{2} \| \nabla f^*(x^*) - \nabla f^*(y^*) \|^2 \\
& = \langle x^*, \nabla f^*(x^*) \rangle - f (\nabla f^*(y^*)) - \langle y^*, \nabla f^*(x^*) - \nabla f^*(y^*) \rangle - \frac{\mu}{2} \| \nabla f^*(x^*) - \nabla f^*(y^*) \|^2 \\
& = \langle y^*, \nabla f^*(y^*) \rangle - f (\nabla f^*(y^*)) + \langle x^* - y^*, \nabla f^*(x^*) \rangle - \frac{\mu}{2} \| \nabla f^*(x^*) - \nabla f^*(y^*) \|^2 \\
& = f^*(y^*) + \langle x^* -y^*, \nabla f^*(x^*) \rangle   - \frac{\mu}{2} \| \nabla f^*(x^*) - \nabla f^*(y^*) \|^2 \\
& = f^*(y^*) + \langle x^* -y^*, \nabla f^*(y^*) \rangle  + \langle x^* -y^*, \nabla f^*(x^*) - \nabla f^*(y^*) \rangle  \\
& \qquad  - \frac{\mu}{2} \| \nabla f^*(x^*) - \nabla f^*(y^*) \|^2 \\
& \le f^*(y^*) + \langle \nabla f^*(y^*) , x^* -y^* \rangle  +  \frac{1}{2 \mu} \| x^* -y^* \|_*^2 
\end{align*}
where the first inequality follows since $f$ is $\mu$-strongly convex (by assumption), and the last inequality by Fenchel's inequality (Lemma \ref{th:fenchel_inequality}). The first and second equalities follow by Lemma \ref{th:conjugate_facts}, the third by rearranging terms, the fourth by Lemma \ref{th:conjugate_facts}, and the fifth by adding $0 = \langle x^* -y^*, \nabla f^*(y^*) - \nabla f^*(y^*) \rangle$ to the right-hand-side and rearranging terms.

$\Longleftarrow$

Since $f^*$ is $1/\mu$-smooth (see $\Longrightarrow$), 
\[
f^*(x^*)  \le f^*(y^*) + \langle \nabla f^*(y^*) , x^* -y^* \rangle  +  \frac{1}{2 \mu} \| x^* -y^* \|_*^2 
\]
Add  $\langle a, x^* \rangle $ to both sides and rearrange terms
\begin{align*}
\langle a, x^* \rangle - \frac{1}{2 \mu} \| x^* -y^* \|_*^2 & \le \langle \nabla f^*(y^*) + a, x^* \rangle  - f^*(x^*) + f^*(y^*) - \langle \nabla f^*(y^*) , y^* \rangle  \\
& = \langle \nabla f^*(y^*) + a, x^* \rangle  - f^*(x^*) -f( \nabla f^*(y^*) )
\end{align*}
where the last equality follows by Lemma \ref{th:conjugate_facts}.
The equation above holds for any $x^*$ and $y^*$. Maximize the left-hand-side and right-hand-side of the last equation w.r.t.\ $x^*$ to obtain
\begin{align*}
\langle a, y^* \rangle + \frac{\mu}{2 } \| a \|^2 & \le f( \nabla f^*(y^*) + a) -f( \nabla f^*(y^*) )
\end{align*}
See \cite[pp.\ 475; Eq.\ 11(3)]{RW98} for the left-hand-side. Rearranging terms yields
\begin{align*}
f( \nabla f^*(y^*) + a)  \ge f( \nabla f^*(y^*) ) + \langle a, y^* \rangle + \frac{\mu}{2} \| a \|^2 
\end{align*}

Finally, let $a = \nabla f^*(x^*) - \nabla f^*(y^*)$ and use the fact that the fact that $y^* =  \nabla f(\nabla f^*(y^*))$ to obtain
\[
f(\nabla f^*(x^*))  \ge f(\nabla f^*(y^*)) + \langle \nabla f(\nabla f^*(y^*)) , \nabla f^*(x^*) - \nabla f^*(y^*) \rangle + \frac{\mu}{2} \| \nabla f^*(x^*) - \nabla f^*(y^*) \|^2
\]
which holds for any  $\nabla f^*(x^*), \nabla f^*(y^*)$ --- recall that $\nabla f^*(x^*), \nabla f^*(y^*) \in C$. 
\end{proof}

\section{Proofs of the Main Results}
\label{eq:main_proof_theorem}

\subsection{Proof of Lemma \ref{th:dual_averaging_bound}}

The proof is divided into two parts. We first present the following lemma. 

\begin{lemma}
\label{th:general_upper_bound}
Let $f : C \to \R^n$ be a $L$-smooth and ($\mu$-strongly) convex function and $C$ a convex set. Suppose Assumption \ref{as:norm} holds and select a $y_0 \in C$ and set $A_0 = 0$. For any sequences $\{x_i \in C\}_{i=1}^k$, $\{y_i \in C\}_{i=1}^k$,  $\{ \alpha_i \ge 0 \}_{i=1}^k$, it holds:
 \begin{align}
 A_k (f(y_k) - f^\star) &    \le 
 N_k  + \sum_{i=1}^k    \langle  \widetilde \nabla f(x_i),  A_i y_i - A_{i-1} y_{i-1} - \alpha_i y^\star \rangle  + \frac{L}{2} \sum_{i=1}^k  A_i   \| y_i - x_i \|^2 \notag \\
& \qquad - \frac{\mu}{2} \sum_{i=1}^k {\alpha_i}\| x_i - y^\star \|^2  - \frac{1}{2L} \sum_{i=1}^k A_{i-1} \| \nabla f(x_i) - \nabla f( y_{i-1}) \|_*^2 \notag
\end{align}
where $N_k \coloneqq - \sum_{i=1}^k    \langle \xi_i,  A_i y_i - A_{i-1} y_{i-1} - \alpha_i y^\star  \rangle$. 
\end{lemma}

\begin{proof}

%

Since $f$ is $L$-smooth and $A_k \ge 0$, 
\[
A _k f(y_k) \le A _k \left(f(x_k) + \langle \nabla f(x_k) , y_k - x_k \rangle  + \frac{L}{2}  \| y_k - x_k \|^2 \right)
\]
Similarly, by the ($\mu$-strong) convexity of $f$, 
\[
- \alpha_k f^\star  \le -  \alpha_k \left( f(x_k) +  \langle \nabla f(x_k), y^\star - x_k \rangle 
 + \frac{\mu}{2}  \| y^\star - x_k \|^2 \right).
\]
Combining the bounds and adding $-A_{k-1} f^\star$ to both sides yields
\begin{align}
A _k f(y_k) - A_k f^\star &  \le A _{k-1} f(x_k) - A_{k-1} f^\star + \langle \nabla f(x_k) , A_k y_k -  A_{k-1} x_k - \alpha_k y^\star \rangle \notag \\
& \quad + \frac{L}{2} A_k  \| y_k - x_k \|^2  - \frac{\mu}{2}  \alpha_k \| y^\star - x_k \|^2
\label{eq:first_up_bound_lemz}
\end{align}

We proceed to upper bound $A_{k-1} f(x_k)$.  Observe
\begin{align}
f(x_k) 
& = \langle x_k, \nabla f(x_k) \rangle - f^*(\nabla f(x_{k})) \notag \\
& \le \langle x_{k}, \nabla f( x_{k}) \rangle - f^*(\nabla f(y_{k-1})) - \langle y_{k-1},  \nabla f(x_k) - \nabla f( y_{k-1})  \rangle \notag \\
& \quad - \frac{1}{2L} \| \nabla f(x_k) - \nabla f( y_{k-1}) \|_*^2 \notag \\
& = f(y_{k-1}) + \langle x_{k}  - y_{k-1}, \nabla f( x_k ) \rangle  - \frac{1}{2L} \| \nabla f(x_k) - \nabla f( y_{k-1}) \|_*^2 \label{eq:coercivity_up}
\end{align}
where the first equality follows by the definition of convex conjugate, the second inequality since $f^*$ is $1/L$-strongly convex, and the last equation again by the definition of the convex conjugate. 
Multiplying the last equation across by $A_{k-1}$ and using the bound in Eq.\ \eqref{eq:first_up_bound_lemz}
\begin{align*}
A _k (f(y_k) -  f^\star) &  \le A _{k-1}( f(y_{k-1}) - f^\star) + \langle \nabla f(x_k) , A_k y_k -  A_{k-1} y_{k-1} - \alpha_k y^\star \rangle  \\
& \quad + \frac{L}{2} A_k  \| y_k - x_k \|^2  - \frac{\mu}{2}  \alpha_k \| y^\star - x_k \|^2 - \frac{A_{k-1}}{2L} \| \nabla f(x_k) - \nabla f( y_{k-1}) \|_*^2 
\end{align*}
Apply the argument recursively from $i=1,\dots,k$ and add $0 =  \sum_{i=1}^k \langle \xi_i - \xi_i , A_i y_i - A_i y_{i-1} - \alpha_i y^\star \rangle $ to the right-hand-side to obtain the result. 
\end{proof}

The second part of the proof consists of applying Nesterov's dual-averaging recursion. Observe that we can write the bound in Lemma  \ref{th:general_upper_bound} as
 \begin{align}
 A_k (f(y_k) - f^\star) &    \le 
\phi(y^\star) +  N_k  + \sum_{i=1}^k    \langle  \widetilde \nabla f(x_i),  A_i y_i - A_{i-1} y_{i-1} \rangle  + \frac{L}{2} \sum_{i=1}^k  A_i   \| y_i - x_i \|^2 \notag \\
& \qquad   - \frac{1}{2L} \sum_{i=1}^k A_{i-1} \| \nabla f(x_i) - \nabla f( y_{i-1}) \|_*^2 + \langle s_k, y^\star \rangle  -\varphi_k(y^\star) \notag
\label{eq:general_upper_bound}
\end{align}
where $\varphi_k (u) = \phi(u) + \frac{\mu}{2} \sum_{i=1}^k \alpha_i \| u - x_i \|^2$ --- recall that $\varphi_k (u)$ is $(\mu A_k + \sigma)$-strongly convex. 

The proof's strategy is inspired on the proof of \cite[Theorem 1]{Nes09} and consists of upper bounding $\langle s_k, y^\star \rangle  -\varphi_k(y^\star)$. From the definition of convex conjugate, we have that
\begin{align*}
&  \langle s_k , y^\star \rangle - \varphi_k(y^\star)    \le \sup_{u \in C}  \{ \langle s_k , u \rangle - \varphi_k(u) \}  = \varphi^*_k(s_k) 
\end{align*}
Next, let $v_k \in \arg\max_{u \in C} \{ \langle s_k, u \rangle - \varphi_k(u) \}$ and observe
\begin{align*}
 \varphi^*_k(s_k) 
&  = \langle s_k, v_k \rangle -\varphi_k(v_k) \\
&  = \langle s_k,v _k \rangle  - \varphi_{k-1} (v_k) - \frac{\mu}{2} \alpha_k \| v_k - x_k \|^2  \\
&  \le \langle s_k,v _k \rangle  - \varphi_{k-1} (v_{k-1}) - \langle \nabla \varphi_{k-1} (v_{k-1}) , v_{k} - v_{k-1} \rangle \\
& \qquad - \frac{\mu A_{k-1} + \sigma}{2} \| v_k - v_{k-1} \|^2  - \frac{\mu}{2} \alpha_k \| v_k - x_k \|^2  \\
&  \le \varphi_{k-1}^*(s_{k-1}) - \langle v_{k}, \alpha_k \widetilde \nabla f(x_k) \rangle - \frac{\mu A_{k} + \sigma}{2} \left\| v_k - \frac{\mu A_{k-1} + \sigma}{\mu A_{k} + \sigma} v_{k-1} - \frac{\mu \alpha_k}{\mu A_{k} + \sigma} x_k \right\|^2 
\end{align*}
where the first inequality follows since $\varphi_{k}$ is $(\mu A_{k} + \sigma)$-strongly convex and the second because:
\begin{align*}
& \langle s_k,v _k \rangle  - \varphi_{k-1} (v_{k-1}) - \langle \nabla \varphi_{k-1} (v_{k-1}) , v_{k} - v_{k-1} \rangle \\
& \qquad =  \langle s_k,v _k \rangle  - \varphi_{k-1} (v_{k-1}) - \langle \nabla \varphi_{k-1} ( \nabla \varphi_{k-1}^* (s_{k-1})) , v_{k} - v_{k-1} \rangle \\
&   \qquad = \langle s_k,v _k \rangle  - \varphi_{k-1} (v_{k-1}) - \langle s_{k-1} , v_{k} - v_{k-1} \rangle \\
&   \qquad = \varphi_{k-1}^* (s_{k-1})  + \langle s_k - s_{k-1},v _k \rangle  \\
&   \qquad = \varphi_{k-1}^* (s_{k-1})  - \langle v _k , \alpha_k \widetilde \nabla f(x_k) \rangle && \text{(by Eq.\ \eqref{eq:v_select})}
\end{align*}
and 
\begin{align*}
& - \frac{\mu A_{k-1} + \sigma}{2} \| v_k - v_{k-1} \|^2  - \frac{\mu}{2} \alpha_k \| v_k - x_k \|^2  \\
& \qquad \le - \frac{\mu A_{k} + \sigma}{2} \left\| v_k - \frac{\mu A_{k-1} + \sigma}{\mu A_{k} + \sigma} v_{k-1} - \frac{\mu \alpha_k}{\mu A_{k} + \sigma} x_k \right\|^2 && \text{(by the convexity of $\| \cdot \|^2$)}\\
& \qquad = - \frac{\mu A_{k} + \sigma}{2} \left\| v_k - \hat v_{k-1} \right\|^2
\end{align*}
Apply the argument above recursively from $i=1,\dots,k$ to obtain
\begin{align*}
\varphi^*_k(s_k)  \le \varphi^*_{0}(s_{0})  - \sum_{i=1}^k \langle \alpha_i \widetilde \nabla f(x_i), v_i \rangle  - \sum_{i=1}^k \frac{\mu A_{i} + \sigma}{2} \left\| v_i - \hat v_{i-1} \right\|^2
\end{align*}
where $\varphi^*_{0}(s_{0}) = \varphi^*_{0}(0) = \sup_{u \in C} - \phi(u) = 0$ since $\phi$ is non-negative by assumption. Plugging the last equation back into the bound of Lemma \ref{th:general_upper_bound} yields the result. 



\subsection{Proof of Theorem \ref{th:main_theorem}}
\label{sec:proof_main_theorem}
We start by presenting the following lemma, which justifies the choice of set $Y_k$.

\begin{lemma}
\label{th:general_y_selection}
For any $x,y,s \in C$ and $L >0$, we have that 
\[
\sup_{\nabla f(x) } \left\{ \min_{y \in C}\left\{ \langle \nabla f(x) , y - s \rangle + \frac{L}{2}\| y -x\|^2 \right\} \right\} \le \frac{L}{2} \| x - s \|^2
\]
\end{lemma}
\begin{proof} Let $w = \nabla f(x)$ and $\psi(y) \coloneqq \frac{L}{2}\| y -x \|^2$. Observe 
\begin{align*}
 \inf_{y \in C} \{ \langle w, y   - s  \rangle + \psi(y) \} 
& = \{ \langle -w, s\rangle -\psi^*(-w)\} \\
& \le \sup_{w} \{ \langle -w, s\rangle -\psi^*(-w)\} \\
& = \psi^{**}(s) \\
&   = \psi(s)
\end{align*}
where the last equation holds since $\psi (\cdot) = \| \cdot \|$ is convex, proper, and lower-semi-continuous; see \cite[Theorem 11.1; pp.\ 474]{RW98}.
\end{proof}

Selecting the sequences $\{x_i \in C\}_{i=1}^k$, $\{y_i\in C\}_{i=1}^k$ as indicated in Algorithm \ref{al:smooth}, the bound in Lemma \ref{th:dual_averaging_bound} can be written as
 \begin{align*}
A_k (f(y_k) - f^\star) & \le  \phi(y^\star) + N_k  - D_k + \frac{1}{2} \sum_{i=1}^k  \left(L \frac{\alpha_i^2}{A_i} - \mu A_{i} - \sigma \right)   \left\|  v_i - \hat v_i   \right\|^2,
\end{align*}
where  $D_k \coloneqq (2L)^{-1} \sum_{i=1}^k A_{i-1} \| \nabla f(x_i) - \nabla f( y_{i-1}) \|_*^2$. We want that the third term in right-hand-side of Eq.\ \eqref{eq:c_bound_acc} is non-positive. For that, we maximize $\alpha_k$ subject to $L \alpha_k^2 - \mu A_{k}^2  - \sigma A_{k} = ( L - \mu ) \alpha_k^2   - (2 \mu A_{k-1} + \sigma)\alpha_k - \mu A_{k-1}^2  - \sigma A_{k-1}   \le 0$. That is, we select
\begin{align}
\alpha_k = \frac{2\mu A_{k-1} + \sigma + \sqrt{(2\mu A_{k-1} + \sigma )^2 + 4 (L-\mu) ( \sigma A_{k-1}  + \mu A^2_{k-1})}}{2(L-\mu)},
\label{eq:quadratic_weight}
\end{align}
and therefore
 \begin{align*}
A_k (f(y_k) - f^\star) & \le  \phi(y^\star) + N_k  - D_k .
\end{align*}
Note that the assumption that $L > \mu$ is necessary because of Eq.\ \eqref{eq:quadratic_weight}.

It remains to lower bound $A_k$. We start when $f$ is not strongly convex (i.e., $\mu = 0$). The choice of $\alpha_k$ satisfies
\[
\alpha^2_k = \frac{\sigma}{L} A_k
\]
Now, we show there exists a sequence $\hat \alpha_k $ such that $\hat \alpha_k^2 \le \frac{\sigma}{L} \hat A_k$ for all $k$ and that  $\hat \alpha_k \le \alpha_k$ and $\hat A_k \le A_k$ for all $k\ge 1$. Let   $\hat \alpha_k = \frac{\sigma}{L}\cdot \frac{(k+1)}{2}$ and observe
\begin{align*}
\hat \alpha_k^2  & = \frac{\sigma^2}{L^2} \cdot \frac{(k+1)^2}{4} \\
& \le \frac{\sigma^2}{L^2} \cdot \frac{(k^2 + 3k)}{4} \\
& = \frac{\sigma^2}{L^2} \cdot \frac{(k+1) k + 2k}{4} \\
& = \frac{\sigma^2}{2L^2} \cdot \frac{(k+1) k}{2} + k \\
& = \frac{\sigma^2}{2L^2 }\sum_{i=1}^k (i+1) \\
& = \frac{\sigma}{L}\sum_{i=1}^k \frac{\sigma}{L} \cdot \frac{(i+1)}{2} \\
& = \frac{\sigma}{L}\sum_{i=1}^k \hat \alpha_k  \\
& = \frac{\sigma}{L} \hat A_k
\end{align*}
The fact that $\hat \alpha_k \le \alpha_k$ follows since $\alpha_k^2 = \frac{\sigma}{L} A_k$ (with equality).  
Finally, since
\begin{align*}
\hat A_k = \frac{\sigma}{2L} \left( \frac{k^2 + 3k}{2} 
\right) = \frac{\sigma}{2L } \left( \frac{(k+1) (k+2)}{2} - 1
\right) 
\end{align*}
and
\begin{align}
\prod_{i=1}^k \left(1 + \frac{2}{i} \right) = \frac{3}{1} \cdot \frac{4}{2} \cdot \frac{5}{3} \cdot \frac{6}{4} \cdot \frac{7}{5} \cdots \frac{(k+1)}{(k-1)} \cdot \frac{(k+2)}{k} = \frac{(k+1)(k+2)}{2}
\label{eq:product_p2}
\end{align}
we obtain $A_k \ge \hat A_k = \frac{\sigma}{2L } \left( \prod_{i=1}^k \left( 1 + \frac{2}{i} \right) - 1
\right) $.

We now consider the strongly convex case. From Eq.\ \eqref{eq:quadratic_weight}, the stated choice of $\alpha_k$ ensures that $\alpha_k \ge \sqrt{\frac{\mu}{L}} A_{k} \ge \sqrt{\frac{\mu}{L}} A_{k-1}$ for all $k\ge1$. And since $\alpha_k  = A_k- A_{k-1}$, 
\[
A_k \ge \left(1 + \sqrt \frac{\mu}{L} \right) A_{k-1}
\]
Applying the argument recursively 
\[
A_k \ge \left(1 + \sqrt \frac{\mu}{L} \right) A_{k-1} \ge \cdots \ge \left(1 + \sqrt \frac{\mu}{L} \right)^{k-1} A_1
\]
Since $A_1 \ge \frac{\sigma}{L}$ by construction, we can lower bound it as follows:
\[
A_1 \ge  \frac{\sigma}{L} = \frac{\sigma}{L} \cdot \frac{(1 + \sqrt \frac{\mu}{L})}{(1 + \sqrt \frac{\mu}{L})} \ge \frac{\sigma}{2L} \left(1 + \sqrt \frac{\mu}{L}\right)
\]
 where the last inequality follows since $ \frac{\mu}{L} \le 1$. Hence, 
\[
A_k \ge \frac{\sigma}{2L}\prod_{i=1}^k \left(1 + \sqrt \frac{\mu}{L} \right)
\]
Finally, combining the two cases 
\begin{align*}
A_k \ge \frac{\sigma}{2L}\prod_{i=1}^k \left(1 + \max \left\{ \frac{2}{i}, \sqrt \frac{\mu}{L} \right\} \right) - \frac{\sigma}{2L}
\end{align*}
which concludes the proof.

\section{Proofs of Section \ref{sec:bounding_nk}}
 
\subsection{Proof of Lemma \ref{th:bound_unbiased}}
Select $y_k = \hat y_k = \frac{A_{k-1}}{A_k} y_{k-1} + \frac{\alpha_k}{A_k} v_{k}$. We have
\begin{align*}
  A_k \left \langle \xi_k ,   \frac{A_{k-1}}{A_k} y_{k-1}  + \frac{\alpha_k}{A_k} y^\star - y_k  \right\rangle  & =   \alpha_k  \left \langle \xi_k ,   y^\star - v_k \right\rangle \\
 & =  \alpha_k  \left \langle \xi_k ,   y^\star - \hat v_k \right\rangle + \alpha_k \langle \xi_k, \hat v_k - v_k \rangle \\
  & \le  \alpha_k  \left \langle \xi_k ,   y^\star - \hat v_k \right\rangle + \alpha_k \|  \xi_k \|_* \| \hat v_k - v_k \| 
 \end{align*}
 where $\hat v_k$ is any vector in $C$. Now, let $\hat v_k \in \arg \max_{u \in C} \{ \langle s_{k-1} - \alpha_k \nabla f(x_k), u \rangle - \varphi_k(u) \} $. Hence, 
 \begin{align*}
 \| \hat v_k - v_k \| & = \left\| \nabla \varphi^*_k \left(s_{k-1} - \alpha_k \nabla f(x_k)  \right) - \nabla \varphi_k^* (s_k) \right\| \\
 & = \left\| \nabla \varphi^*_k \left(s_{k} + \alpha_k \xi_k\right)- \nabla \varphi_k^* (s_k) \right\| \\
 & \le \frac{1}{\mu A_k + \sigma }\| \alpha_k \xi_k \|_*
 \end{align*}
 where the last equation follows since $\varphi^*_k$ has $\frac{1}{\mu A_k + \sigma}$-Lipschitz continuous gradients. Hence, 
 \[
 A_k \left \langle \xi_k ,   \frac{A_{k-1}}{A_k} y_{k-1}  + \frac{\alpha_k}{A_k} y^\star - y_k  \right\rangle  \le \alpha_k  \left \langle \xi_k ,   y^\star - \hat v_k \right\rangle + \frac{\alpha_k^2}{\mu A_k + \sigma }\|  \xi_k \|_*^2
 \] 
 Taking expectations and using the fact that $\xi_k$ is independent of $y^\star$ and $\hat v_k$, we have  
  \[
\mathbf E \left[ A_k \left \langle \xi_k ,   \frac{A_{k-1}}{A_k} y_{k-1}  + \frac{\alpha_k}{A_k} y^\star - y_k  \right\rangle \right]
 \le \frac{\alpha_k^2}{\mu A_k + \sigma } \mathbf E [\|  \xi_k \|^2_*] 
 \] 
 Summing from $i=1,\dots,k$ and using the fact that $ \frac{\alpha_k^2}{\mu A_k + \sigma }  = \lambda A_k$, we obtain the stated result. 

\subsection{Proof of Corollary \ref{th:prop_unconstrained}}
Recall
\begin{align*}
v_k & = \arg \max_{u \in C} \{ \langle s_k, u \rangle - \varphi_k(u) \} \\
&  =  \arg \max_{u \in C} \left\{  \left\langle s_k , u\right\rangle - \frac{1}{2} \| u  \|_2^2 -  \frac{\mu}{2} \sum_{i=1}^k \alpha_i \| u - x_i\|_2^2 \right\}
\end{align*}
Since $C =  \R^n$
\begin{align*}
\nabla_u \left( \left\langle s_k , u\right\rangle - \frac{1}{2} \| u  \|_2^2 -  \frac{\mu}{2} \sum_{i=1}^k \alpha_i \| u - x_i\|_2^2 \right) 
& = s_k - u - \mu \sum_{i=1}^k \alpha_i (u - x_i) \\
&  = s_k - u - \mu A_k u + \mu \sum_{i=1}^k \alpha_i x_i
\end{align*}
Equating the last equation to zero, letting $v_k = u$, and rearranging terms yields
\[
v_k = \frac{s_k + \mu \sum_{i=1}^k \alpha_i x_i}{\mu A_k + \sigma}
\]
where $\sigma = 1$. Hence,
\begin{align*}
v_k - \hat v_{k-1} & = v_k - \frac{\mu A_{k-1} + \sigma}{\mu A_{k} + \sigma} v_{k-1} - \frac{\mu \alpha_k}{\mu A_{k} + \sigma} x_k\\
&  =  \frac{s_k + \mu \sum_{i=1}^k \alpha_i x_i}{\mu A_k + \sigma} - \frac{s_{k-1} + \mu \sum_{i=1}^{k-1} \alpha_i x_i}{\mu A_{k} + \sigma} - \frac{\mu \alpha_k x_k}{\mu A_{k} + \sigma} \\
&  = \frac{s_k - s_{k-1}}{\mu A_k  + \sigma} \\
&  = - \frac{\alpha_k \widetilde \nabla f(x_k)}{\mu A_k + \sigma}
\end{align*}
and therefore
\begin{align}
 A_k (f(y_k) - f^\star)  &  \le \phi(y^\star) + \frac{1}{2}\sum_{i=1}^k  \frac{\alpha_i^2}{\mu A_{i} + \sigma} \left( 2 \mathbf E [ \| \xi_i \|_*^2 ] - (1-\lambda)\left\| \widetilde \nabla f(x_i) \right\|^2  \right)    \notag
 \end{align}
Observe
\begin{align*}
\mathbf E \left[ \left\| \widetilde \nabla f(x_k) \right\|_2^2 \right]
& = \mathbf E \left[  \left\|  \nabla f(x_k) + \xi_k \right\|_2^2 \right] \\
& = \mathbf E \left[  \|  \nabla f(x_k) \|_2^2 + \| \xi_k \|_2^2 + 2 \langle \nabla f(x_k),\xi_k \rangle \right] \\
& =    \left\|  \nabla f(x_k) \right\|_2^2  + \mathbf E \left[ \left\| \xi_k \right\|_2^2 \right]
\end{align*}
where the last inequality follows since $\xi_k$ and $\nabla f(x_k)$ are independent.
Hence, 
\begin{align*}
2 \mathbf E [ \| \xi_k \|_2^2 ] - (1-\lambda)\left\| \widetilde \nabla f(x_k) \right\|_2^2  = (1+\lambda) \mathbf E [ \| \xi_k \|_2^2 ] - (1-\lambda) \| \nabla f(x_k)\|_2^2
\end{align*}
and therefore
\begin{align}
 A_k (f(y_k) - f^\star)  &  \le \phi(y^\star) + \frac{(1+\lambda)}{2}\sum_{i=1}^k  \frac{\alpha_i^2}{\mu A_{i} + \sigma} \left( \mathbf E [ \| \xi_k \|_2^2 ] - \frac{(1-\lambda)}{(1+\lambda)}\left\|  \nabla f(x_k) \right\|_2^2  \right)    \notag
 \end{align}
%
%
%


\subsubsection{Proof of Lemma \ref{th:prop_constrained}}
We start by presenting some lemmas. 

\begin{lemma}
\label{th:squared_norm_upper_bound}
$\| x + y \|^2 \le (1+\gamma^{-1})\| x \|^2 + (1+\gamma) \| y \|^2$ with $\gamma > 0$. 
\end{lemma}
\begin{proof}

\begin{align*}
\| x + y \|^2 & \le (\| x\| + \| y \|)^2 && \text{(By Minkowski's ineq.)}\\
& = \| x\|^2 + \| y \|^2 + 2 \| x\| \|y \| \\
& \le \| x\|^2 + \| y \|^2 + 2 \| x\| \|y \| + \frac{1}{\gamma} ( \| x \| - \| \gamma y \|)^2 \\
& = \| x\|^2 + \| y \|^2 + 2 \| x\| \|y \| + \frac{1}{\gamma} \| x \|^2 + \gamma \| y \|^2 - 2  \| x \| \| y \| \\
& = (1+ \gamma^{-1}) \| x\|^2 + (1 + \gamma) \| y \|^2 
\end{align*}
\end{proof}

In the following, we consider the case where the objective is the sum of loss functions, as this will allow us to streamline the results in Sec.\ \ref{sec:sup_app_saga}. That is, we have
\[
f(x) = \sum_{l=1}^m f_l(x)
\]
where $f_l$ has $\frac{L}{m}$-Lipschitz continuous gradient. Note that  $f$ is $L$-smooth.

\begin{lemma} \label{th:intermediate_lemma_lm}
Let $f(x) = \sum_{l=1}^m f_l(x)$ where each $f_l$ has $\frac{L}{m}$-Lipschitz continuous gradient. Select $\lambda \le \frac{1}{m+1}$.  Then, 
\begin{align*}
& - (1-\lambda) \frac{(\mu A_{k} + \sigma)}{2} \left\| v_k - \frac{\mu A_{k-1} + \sigma}{\mu A_{k} + \sigma} v_{k-1} - \frac{\mu \alpha_k}{\mu A_{k} + \sigma} x_k \right\|^2 \\
& \qquad \le - \frac{mA_{k} }{2L}  \sum_{l=1}^m \| \nabla f_l (x_k) - \nabla f_l (y_k) \|_*^2  
\end{align*}
\end{lemma}
\begin{proof}
Since $f_l$, $l=1,\dots,m$ have $\frac{L}{m}$-Lipschitz continuous gradient:
\begin{align*}
\sum_{l=1}^m \| \nabla f_l(x_k) - \nabla f_l(y_k) \|_* 
 \le L \| x_k - y_k \|  \le L \frac{\alpha_k}{A_k} \left\| v_k - \frac{\mu A_{k-1} + \sigma}{\mu A_{k} + \sigma} v_{k-1} - \frac{\mu \alpha_k}{\mu A_{k} + \sigma} x_k \right\|
\end{align*}
where the last inequality follows because of the choice of sequences ensures that $\| \hat y_k - x_k \|^2 \le \alpha_k^2 A_k^{-2} \| v_k -   \frac{\mu A_{k-1} + \sigma}{\mu A_{k} + \sigma} v_{k-1} - \frac{\mu \alpha_k}{\mu A_{k} + \sigma} x_k \|^2$. Hence, 

\[
\frac{1}{L^2 } \cdot \frac{A_k^2}{\alpha_k^2} \sum_{l=1}^m \| \nabla f_l(x_k) - \nabla f_l(y_k) \|_*^2 \le    \left\| v_k - \frac{\mu A_{k-1} + \sigma}{\mu A_{k} + \sigma} v_{k-1} - \frac{\mu \alpha_k}{\mu A_{k} + \sigma} x_k \right\|^2
\]
The choice of $\alpha_k$ ensures that $L \frac{\alpha_k^2}{A_k} = \lambda (\mu A_{k} + \sigma)$, i.e., $\left(\frac{1-\lambda}{\lambda}\right) \frac{L}{2} \frac{\alpha_k^2}{A_k} = (1-\lambda) \frac{(\mu A_{k} + \sigma)}{2}$. Hence, 
\begin{align*}
& \left(\frac{1-\lambda}{\lambda} \right) \frac{A_{k} }{2L}  \sum_{l=1}^m \| \nabla f_l(x_k) - \nabla f_l(y_k) \|_*^2 \\
& \qquad  \le (1-\lambda) \frac{\mu A_{k} + \sigma}{2} \left\| v_k - \frac{\mu A_{k-1} + \sigma}{\mu A_{k} + \sigma} v_{k-1} - \frac{\mu \alpha_k}{\mu A_{k} + \sigma} x_k \right\|^2
\end{align*}
Finally, since $\lambda \le \frac{1}{m+1}$, we have that $\left(\frac{1-\lambda}{\lambda}\right) \ge m$ and therefore the stated result.  
\end{proof}

\begin{lemma} \label{th:dk_separable} Let $f(x_k) = \sum_{l=1}^m f_l(x_k)$ where each $f_l$ has $\frac{L}{m}$-Lipschitz continuous gradient. We have
 \begin{align}
A_k (f(y_k) - f^\star) &  \le \phi(y^\star) + N_k - \sum_{i=1}^k \sum_{l=1}^m \frac{m A_{i-1}}{2L}  \| \nabla f_l(x_i) - \nabla f_l ( y_{i-1}) \|_*^2 \\
& \qquad - (1-\lambda) \sum_{i=1}^k (\mu A_i + \sigma) \| v_i - \hat v_{i-1} \|^2
 \notag
\end{align}
\end{lemma}
\begin{proof}
For the result, we just need to repeat the step in Eq.\ \eqref{eq:coercivity_up}. By the $\frac{L}{m}$-smoothness of $f_l$, we have that
\begin{align*}
f_l(x_k) \le f_l(y_{k-1}) + \langle x_k - y_{k-1} , \nabla f_l(x_k) \rangle - \frac{m A_{k-1}}{2L} \| \nabla f_l(x_k) - \nabla f_l(y_{k-1}) \|_*^2
\end{align*}
Summing from $l=1,\dots,m$ and using the fact that $f(x_k) = \sum_{l=1}^m f_l(x_k)$, 
\[
f(x_k) \le f(y_{k-1}) + \langle x_k - y_{k-1}, \nabla f(x_k) \rangle - \frac{m A_{k-1}}{2L} \sum_{l=1}^m \| \nabla f_l(x_k) - \nabla f_l(y_{k-1}) \|_*^2
\]
The rest of the proof in Lemma \ref{th:dual_averaging_bound} remains the same. 
\end{proof}


\begin{lemma}
\label{th:compact_negative_term}
Consider the setup of Lemma \ref{th:intermediate_lemma_lm} and select $\lambda \le \frac{1}{2}$. Then,
\begin{align*}
 A_k (f(y_k) - f^\star) 
& \le \phi(y^\star) + N_k  - \frac{m}{4L} \sum_{i=1}^k  \sum_{l=1}^m  A_{i-1}  \| \nabla f_l(x_i) - \nabla f_l(x_{i-1}) \|_*^2 
 \end{align*}
\end{lemma}
\begin{proof}
By Lemmas \ref{th:intermediate_lemma_lm} 
 \begin{align*}
&  - (1-\lambda) \frac{\mu A_{k} + \sigma}{2} \left\| v_k - \frac{\mu A_{k-1} + \sigma}{\mu A_{k} + \sigma} v_{k-1} - \frac{\mu \alpha_k}{\mu A_{k} + \sigma} x_k \right\|  -\frac{m A_{k-1}}{2L} \sum_{l=1}^m \| \nabla f_l(x_k) - \nabla f_l(y_{k-1}) \|_*^2 \\
& \qquad \le 
- \frac{mA_{k} }{2L}  \sum_{l=1}^m \| \nabla f_l (x_k) - \nabla f_l (y_k) \|_*^2  - \frac{m A_{k-1} }{2L} \sum_{l=1}^m \| \nabla f_l(x_k) - \nabla f_l( y_{k-1}) \|_*^2
\end{align*}
Summing from $i=1,\dots,k$ 
\begin{align*}
& - \frac{m}{2L}  \sum_{i=1}^k \left( A_{i} \sum_{l=1}^m \| \nabla f_l (x_i) - \nabla f_l (y_i) \|_*^2  + A_{i-1} \sum_{l=1}^m \| \nabla f_l(x_i) - \nabla f_l( y_{i-1}) \|_*^2 \right) \\
& \qquad = - \frac{1}{2L} \sum_{i=1}^{k} \left( A_{i-1}  \sum_{l=1}^m \| \nabla f_l(x_{i-1}) - \nabla f_l(y_{i-1}) \|_*^2 +  A_{i-1}  \sum_{l=1}^m \| \nabla f_l(x_i) - \nabla f_l( y_{i-1}) \|_*^2   \right) \\
& \qquad \qquad - \frac{ m A_{k}}{2L}  \sum_{l=1}^m \| \nabla f_l(x_k) - \nabla f_l( y_{k}) \|_*^2 + \frac{A_{0}}{2L}  \sum_{l=1}^m \| \nabla f_l(x_0) - \nabla f_l( y_{0}) \|_*^2 \\
& \qquad \le - \frac{m}{2L} \sum_{i=1}^{k}A_{i-1}  \left(   \sum_{l=1}^m \left( \| \nabla f_l(x_{i-1}) - \nabla f_l(y_{i-1}) \|_*^2 +  \| \nabla f_l(x_i) - \nabla f_l( y_{i-1}) \|_*^2 \right)  \right) \\
& \qquad \le - \frac{m}{4L} \sum_{i=1}^{k}   \sum_{l=1}^m A_{i-1}  \| \nabla f_l(x_{i}) - \nabla f_l(x_{i-1}) \|_*^2  
\end{align*}
where the last inequality follows since by Lemma \ref{th:squared_norm_upper_bound} $\frac{1}{2}\| a - b \|^2 \le \| a\|^2 + \| b\|^2$.
\end{proof}

We are now ready to prove Proposition \ref{th:prop_constrained}. Set  $m=1$. Using Theorem \ref{th:main_theorem} and Lemmas \ref{th:bound_unbiased} and \ref{th:compact_negative_term} with $\lambda \le 1/2$, we have
\begin{align*}
 A_k (f(y_k) - f^\star) 
& \le \phi(y^\star) + \frac{1}{L} \sum_{i=1}^k \lambda A_i \mathbf E [ \| \xi_i \|_*^2 ]  - \frac{1}{4L} \sum_{i=1}^k 
  A_{i-1}  \| \nabla f(x_i) - \nabla f(x_{i-1}) \|_*^2 
\end{align*}
The sufficient condition follows by rearranging terms.

\section{Finite-sum problems with SAGA}
\label{sec:sup_app_saga}

The following result shows that we can use the SAGA oracle  in Algorithm \ref{al:smooth} and retain the accelerated rates. 

\begin{proposition}[Accelerated \textsc{SAGA}] \label{th:saga} Consider the setup of Theorem \ref{th:main_theorem} where $f(x) =  \sum_{l=1}^m f_l(x)$ and each $f_l(x)$ is ${\frac{L}{m}}$-Lipschitz continuous and ($\mu$-strongly) convex. Consider an inexact oracle ${\widetilde {\mathcal O}}$ that returns a vector ${\widetilde \nabla f(x_k)}$ as indicated in Eq.\ \eqref{eq:saga_update}. Then,  Algorithm \ref{al:smooth} ensures
\begin{align}
 A_k \mathbf E[f(y_k) - f^\star]
& \le \phi(y^\star) +\sum_{i=1}^k  \sum_{l=1}^m   \frac{mA_{i-1} }{4L}  \left(\lambda \frac{96 m^2 }{b^3}  - 1\right)   \mathbf E[  \| \nabla f_l(x_i) - \nabla f_l(x_{i-1}) \|_*^2 ] \label{eq:saga_bound}
\end{align}
where $0 <\lambda \le \min \left\{ \frac{1}{m+1}, \frac{L}{\mu} \cdot \frac{b^2}{16m^2} \right\}$ and $A_k \ge \frac{\lambda \sigma}{2L}\prod_{i=1}^k (1 + \max \{ \frac{2}{i}, \sqrt {\frac{\lambda \mu}{L} } \} ) - \frac{\lambda \sigma}{2L}$. 
\end{proposition}

Recall that gradients are estimated as in Eq.\ \eqref{eq:saga_update}. The following lemma upper bounds $N_k$ for this particular type of oracle.

\begin{lemma}
\label{th:saga_supp}
 Let $f(x) = \sum_{l=1}^m f_l(x)$ where each $f_l$ is continuous, and consider that gradients are estimated as in Eq.\ \eqref{eq:saga_update}. We have
\[
N_k \le \frac{\lambda}{L} \cdot  \frac{24 m^3}{b^3} \sum_{i=1}^k \sum_{l=1}^m A_{i-1}  \mathbf E [ \| \nabla f_l(x_i) - \nabla f_l(x_{i-1}) \|_*^2]
\]
where $\lambda \le \min \left\{ \frac{1}{m+1}, \frac{L}{\mu} \cdot \frac{b^2}{16m^2} \right\}$.

\end{lemma}
\begin{proof}
Set $x_0 \in C$ and $\psi_0^l = x_0$ for all $l \in \{ 1,\dots,m \}$.
In each iteration $k \ge 1$,  (i) select a collection of indexes $J_k \subseteq \{1,\dots,n\}$ uniformly at random with $|J_k| = b$, (ii) set $\psi_k^j = x_k$, and (iii) let 
\begin{align*}
\widetilde \nabla f(x_k)  = \frac{m}{b}\sum_{j \in J_k }  \nabla f_j (x_k) - \frac{m}{b} \sum_{j \in J_k } \nabla f_j (\psi_{k-1}^j) + \sum_{l=1}^m \nabla f_l(\psi_{k-1}^l)
\end{align*}
Note that $\mathbf E_{J_k} [ \widetilde \nabla f(x_k) ] = \nabla f(x_k)$ because $\mathbf E_{j \sim \{1,\dots,m\}} [ \nabla f_j (x_k)] =  \frac{1}{m} \sum_{l=1}^m \nabla f_l(x_k)$.
Now, observe
\begin{align*}
&  \mathbf E [ \| \widetilde \nabla f(x_k) - \nabla f(x_k) \|_*^2 ]  \\
 & \qquad = \mathbf E \textstyle [ \| \frac{m}{b} \sum_{j \in J_k } ( \nabla f_j (x_k) - \nabla f_j (\psi_{k-1}^j) )  +  \sum_{l=1}^m \nabla f_l(\psi_{k-1}^l) - \nabla f(x_k) \|_*^2 ] \\
& \qquad = \mathbf E [ \| \textstyle \frac{m}{b} \sum_{j \in J_k } ( \nabla f_j (x_k) - \nabla f_j (\psi_{k-1}^j) ) + \mathbf E [  \frac{m}{b} \sum_{j \in J_k } (\nabla f_j (\psi_{k-1}^j)-  \nabla f_j (x_k)) ]  \|_*^2] \\
& \qquad = \mathbf E [ \| \textstyle \frac{m}{b} \sum_{j \in J_k } (\nabla f_j (\psi_{k-1}^j)-  \nabla f_j (x_k)) \|_*^2] -  \| \mathbf E [ \frac{m}{b} \sum_{j \in J_k } (\nabla f_j (\psi_{k-1}^j)-  \nabla f_j (x_k)) ]  \| _*^2 \\
& \qquad \le \mathbf E [ \|  \textstyle \frac{m}{b} \sum_{j \in J_k } (\nabla f_j (\psi_{k-1}^j)-  \nabla f_j (x_k)) \|_*^2] \\ 
& \qquad = \frac{m^2}{b^2} \mathbf E [ \|  \textstyle  \sum_{j \in J_k } (\nabla f_j (\psi_{k-1}^j)-  \nabla f_j (x_k)) \|_*^2] \\ 
& \qquad = \frac{m^2}{b^2}  \textstyle  \sum_{j \in J_k } \mathbf E [ \| \nabla f_j (\psi_{k-1}^j)-  \nabla f_j (x_k) \|_*^2] \\ 
& \qquad =\textstyle  \frac{m}{b}  \sum_{l=1}^m  \| \nabla f_l (x_k) - \nabla f_l (\psi_{k-1}^l) \|_*^2
\end{align*}
We now proceed to upper bound $\| \nabla f_l (x_k) - \nabla f_l (\psi_{k-1}^l) \|_*^2$. Observe
\begin{align*}
 \| \nabla f_l (x_k) - \nabla f_l (\psi_{k-1}^l) \|_*^2 & \le \left(1+ \frac{2m}{b}\right) \|  \nabla f_l (x_k) - \nabla f_l (x_{k-1}) \|_*^2 \\
& \quad + \left( 1 + \frac{b}{2m} \right) \| \nabla f_l(x_{k-1}) - \nabla f_l (\psi_{k-1}^l) \|_*^2
\end{align*}
By the stated choice of index, we have that $\| \nabla f_l(x_{k-1}) - \nabla f_l (\psi_{k-1}^l) \|^2 = 0$ for the indexes in $J_k$. Hence, taking expectations w.r.t.\ $J_{k-1}$ we have
\[
 \sum_{l=1}^m \mathbf E[\| \nabla f_l(x_{k-1}) - \nabla f_l (\psi_{k-1}^l) \|_*^2 ]  = \left(1-\frac{b}{m} \right) \sum_{l=1}^m \| \nabla f_l (x_{k-1}) - \nabla f_i (\psi_{k-2}^l) \|_*^2
\]
Hence, 
\begin{align*}
&  \mathbf E [ \| \widetilde \nabla f(x_k) - \nabla f(x_k) \|_*^2 ]  \\
& \quad \le \frac{m}{b} \left(1+\frac{2m}{b}\right) \sum_{l=1}^m \mathbf E [ \|  \nabla f_l (x_k) - \nabla f_l (x_{k-1}) \|_*^2 ]  \\
& \qquad +   \left(1-\frac{b}{m} \right) \left(1+\frac{b}{2m}\right) \frac{m}{b} \sum_{l=1}^m \| \nabla f_l (x_{k-1}) - \nabla f_l (\psi_{k-2}^l) \|_*^2  \\
  & \quad \le  \frac{3m^2}{b^2} \sum_{l=1}^m \mathbf E [ \|  \nabla f_l (x_k) - \nabla f_l (x_{k-1}) \|_*^2  ]+   \left(1-\frac{b}{2m} \right)  \frac{m}{b} \sum_{l=1}^m \| \nabla f_l (x_{k-1}) - \nabla f_l (\psi_{k-2}^l) \|_*^2  
  \end{align*}
  where the first inequality follows by Lemma \ref{th:squared_norm_upper_bound} with $\gamma = \frac{b}{2m}$ and the second because $\frac{m}{b} \left(1+\frac{2m}{b}\right) \le \frac{3m^2}{b^2}$ and $ \left(1-\frac{b}{m} \right) \left(1+\frac{b}{2m}\right) \le  \left(1-\frac{b}{2m} \right)$. Apply the argument recursively to obtain
  \begin{align*}
  \mathbf E [ \| \widetilde \nabla f(x_k) - \nabla f(x_k) \|_*^2 ]  
  &  \le   \frac{3m^2}{b^2}  \sum_{\tau=1}^k \left(1-\frac{b}{2m} \right)^{k - \tau}  \sum_{l=1}^m \mathbf E [\| \nabla f_l (x_{\tau}) - \nabla f_l (x_{\tau-1}) \|_*^2 ] \\
  & \qquad +   \left(1-\frac{b}{2m} \right)^k  \frac{m}{b} \sum_{l=1}^m \| \nabla f_l(x_{0}) - \nabla f_l (\psi_{-1}^l) \|_*^2 \\
 &  =   \frac{3m^2}{b^2}  \sum_{\tau=1}^k \left(1-\frac{b}{2m} \right)^{k - \tau}  \sum_{l=1}^m \mathbf E [ \| \nabla f_l(x_{\tau}) - \nabla f_l (x_{\tau-1}) \|_*^2 ]
\end{align*}
We want this to hold
\begin{align}
A_{k} \left( 1 - \frac{b}{2m} \right) \le A_{k-1} \left( 1 - \frac{b}{4m} \right)
\label{eq:absorption_condition}
\end{align}
Recall that $A_k \ge A_{k-1} \ge \cdots \ge A_1$ and $L \frac{ \alpha_k^2}{A_k} = \lambda (\mu A_k + \sigma)$. Hence, when $f$ is $\mu$-strongly convex:
\begin{align*}
\frac{\alpha_k}{A_k} \ge \sqrt{ \frac{\lambda \mu}{L}} \implies \frac{A_k - A_{k-1}}{A_{k-1}} \ge \sqrt{ \frac{\lambda \mu}{L}} \implies A_k \ge \left(1 + \sqrt{ \frac{\lambda \mu}{L}} \right) A_{k-1}
\end{align*}
Similarly, when $f$ is just smooth, we have
\begin{align*}
\frac{\alpha^2_k}{A_k} \ge { \frac{\lambda \sigma}{L}} \implies \frac{(A_k - A_{k-1})^2}{A_{k-1}} \ge { \frac{\lambda \sigma}{L}} \implies A_k \ge \left(1 + \sqrt{ \frac{\lambda \sigma}{LA_{k-1}}} \right) A_{k-1}
\end{align*}
We also want that 
\[
 \left(1 + \sqrt{ \frac{\lambda \mu}{L}} \right) \left( 1 - \frac{b}{2m} \right) \le  \left( 1 - \frac{b}{4m} \right)
 \]
which will hold when $ \left(1 + \sqrt{ \frac{\lambda \mu}{L}} \right)  \le (1 + \frac{b}{4m})$. That is, when $\lambda \le \frac{L}{\mu} \cdot \frac{b^2}{16 m^2}$. The requirement when $f$ is just convex is also that   $\lambda \le \frac{L}{\sigma} \cdot \frac{b^2}{16 m^2} A_{k-1}$. Since if $A_{k-1} \ge \frac{\sigma}{L}$, we have that $\lambda \le \frac{b^2}{16m^2}$. 

Summing from $i=1,\dots,k$
\begin{align*}
& \frac{\lambda}{L}\sum_{i=1}^k A_i \mathbf E [ \| \widetilde \nabla f(x_i) - \nabla f(x_i) \|_*^2 ]  \\
 & \qquad  \le \frac{\lambda}{L} \cdot \frac{3m^2}{b^2}   \sum_{i=1}^k A_i  \sum_{\tau=1}^i \left(1-\frac{b}{2m} \right)^{i - \tau}  \sum_{l=1}^m \mathbf E [ \| \nabla f_l (x_{\tau}) - \nabla f_l (x_{\tau-1}) \|_*^2 ] \\
 & \qquad \le \frac{\lambda}{L} \cdot \frac{3m^2}{b^2}  \sum_{i=1}^k \sum_{\tau=1}^i A_\tau \left(1-\frac{b}{4m} \right)^{i - \tau}  \sum_{l=1}^m \mathbf E [ \| \nabla f_l (x_{\tau}) - \nabla f_l (x_{\tau-1}) \|_*^2 ] \\
& \qquad \le \frac{\lambda}{L} \cdot  \frac{12 m^3}{b^3} \sum_{i=1}^k  A_i \sum_{l=1}^m \mathbf E [ \| \nabla f_l(x_i) - \nabla f_l(x_{i-1}) \|_*^2] \\
& \qquad = \frac{\lambda}{L} \cdot  \frac{12 m^3}{b^3} \sum_{i=1}^k  \frac{A_i}{A_{i-1}} A_{i-1} \sum_{l=1}^m \mathbf E [ \| \nabla f_l(x_i) - \nabla f_l(x_{i-1}) \|_*^2] \\
& \qquad \le \frac{\lambda}{L} \cdot  \frac{24 m^3}{b^3} \sum_{i=1}^k  A_{i-1} \sum_{l=1}^m \mathbf E [ \| \nabla f_l(x_i) - \nabla f_l(x_{i-1}) \|_*^2] 
\end{align*}
where the last equation follows since from Eq.\ \eqref{eq:absorption_condition} 
\[
\frac{A_k}{A_{k-1}} \le \frac{(1-\frac{b}{4m})}{(1 - \frac{b}{2m})} = \frac{4m-b}{4m-2b} \le \frac{4m-b}{2m} \le \frac{4m}{2m} \le 2
\] 
\end{proof}
Combining Lemma \ref{th:compact_negative_term} and \ref{th:saga_supp}, we obtain
\begin{align}
 A_k (f(y_k) - f^\star) 
& \le \phi(y^\star) +\sum_{i=1}^k  \sum_{l=1}^m   \frac{mA_{i-1} }{4L}   \left(\lambda \frac{96 m^2 }{b^3}  - 1\right) \mathbf E[ \| \nabla f_l(x_i) - \nabla f_l(x_{i-1}) \|_*^2 ]
\label{eq:saga_appendix}
\end{align}
where $\lambda \le \min \left\{ \frac{1}{m+1}, \frac{L}{\mu} \cdot \frac{b^2}{16m^2} \right\}$. Selecting $\lambda$ sufficiently small, the second term in the right-hand-side of the last equation is non-positive, and so we can drop it.  That is, we have that $f(y_k) - f^\star 
 \le \frac{\phi(y^\star)}{A_k}$ where $A_k$ is as in Theorem \ref{th:main_theorem}.

\section{Experiments}

This section contains two numerical experiments. A federated logistic regression problem with compressed gradients, and a finite-sum problem with a SAGA oracle. 

%

The optimal value $\smash{\hat f^{\star}}$ in all experiments is computed with a high-accuracy interior-point method in Julia \cite{Julia, IPOPT}. 
%

\subsection{Federated logistic regression with compressed gradients}
\label{sec:app_federated}


We consider the distributed logistic regression problem in the federated setting with the LIBSVM's \emph{mushroom} and \emph{a5a} data sets \cite{libsvm} (BSD license). In short, there are $m=10$ clients, each with a private dataset $\mathcal D_l$, $l=1,\dots,m$ and a loss function $f_l(x) = \sum_{(a,b) \in \mathcal D_l} \log \left( 1 + \exp(-b \langle a, x\rangle) \right)$, where $(a,b)$ are data samples, $a  \in \R^n$, $b \in \{-1,1\}$. The objective function is the sum of the clients' loss function plus a convex regularizer: $\min_{x \in \R^n} \sum_{l =1}^m   f_l(x)+ \frac{1}{2} \| x\|^2_2$.
The federated training process consists of running the accelerated algorithms at a central node and querying the clients for gradients. The clients send to the server a compressed/quantized version of the gradient (i.e., $\widetilde \nabla f_l(x_k)$) to reduce the number of communicated bits (i.e., the time to compute a gradient in a distributed manner). The gradient compression is carried out by three different schemes: na\"ive random sparsification \cite[Sec.\ 3.1]{LKQ+20},\footnote{ The ``na\"ive random sparsification'' gradient compression scheme selects components uniformly at random without  taking into account their magnitude or importance. The latter is essentially what makes ``na\"ive random sparsification'' worst than the two other schemes. } random dithering \cite{AGL+17}, and natural compression \cite{{HHH+20}}.\footnote{See supplementary material for details.} The experiments follow the setup in \cite[Sec.\ 6]{LKQ+20}, and the number of transmitted bits per gradient is as indicated in \cite[Sec.\ 6]{LKQ+20}.

Figures \ref{fig:logsitic_mushroom_appendix} and \ref{fig:logsitic_a5a_appendix} shows the algorithms' performance for both datasets as a function of the number of rounds ($k$) and communicated bits.\footnote{Parameter $\lambda$ in Algorithm\ \ref{al:smooth} satisfies the condition in Proposition \ref{th:prop_unconstrained}.} 
Observe that using an inexact oracle (i.e., compressed gradients) degrades the algorithm's performance in terms of iterations (i.e., dashed line above the solid line in subfigures \ref{fig:logsitic_mushroom_appendix}d, \ref{fig:logsitic_mushroom_appendix}e, \ref{fig:logsitic_mushroom_appendix}f, \ref{fig:logsitic_a5a_appendix}d, \ref{fig:logsitic_a5a_appendix}e, \ref{fig:logsitic_a5a_appendix}f), i.e., using compressed gradients requires computing a larger number of gradients to achieve the same training loss. However, since gradients can be estimated/communicated faster with an inexact oracle, that reduces the total training time.\footnote{The total training time is the total communicated bits divided by the bandwidth available between the clients and the server. }
Also, note the three compression schemes perform differently, with the na\"ive random sparsification being the worst. Finally, note from all plots that the proposed algorithms perform better than ACGD algorithm proposed in \cite{LKQ+20}, which is specifically designed for these types of problems.


\begin{figure*}[h!]
\centering
{\resizebox{0.9\textwidth}{!}{\input{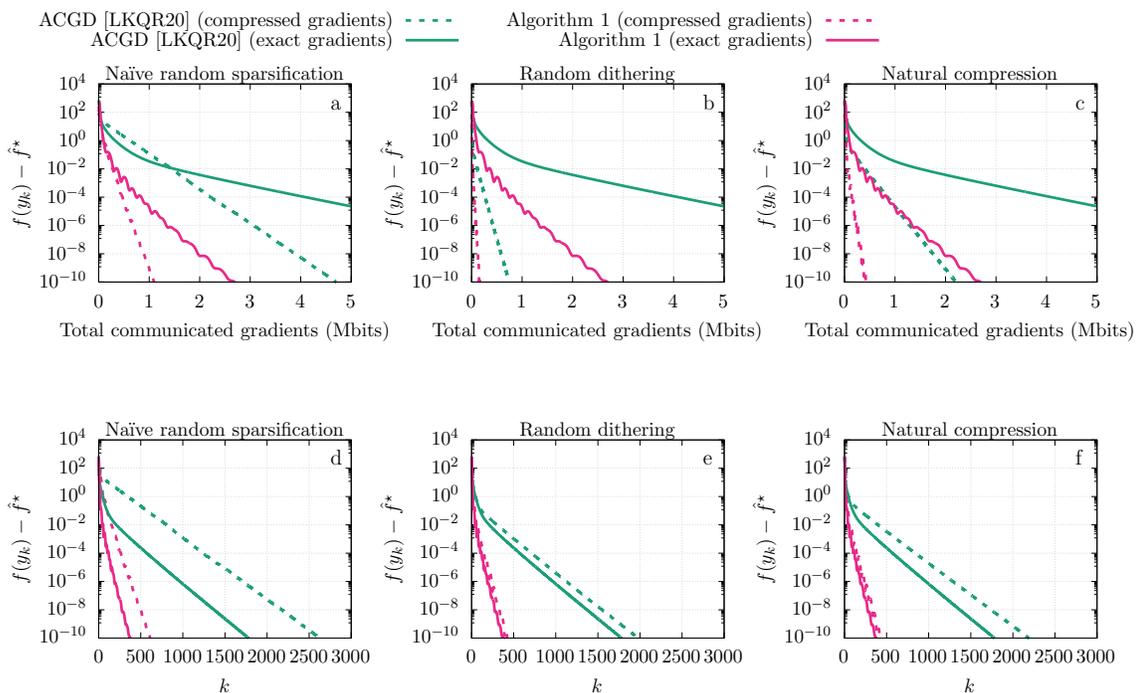}}} 
\caption{Federated logistic regression on the LIBSVM's mushroom dataset with $m = 10$ clients.}
\label{fig:logsitic_mushroom_appendix}
\end{figure*}

\begin{figure*}[h!]
\centering
{\resizebox{0.9\textwidth}{!}{\input{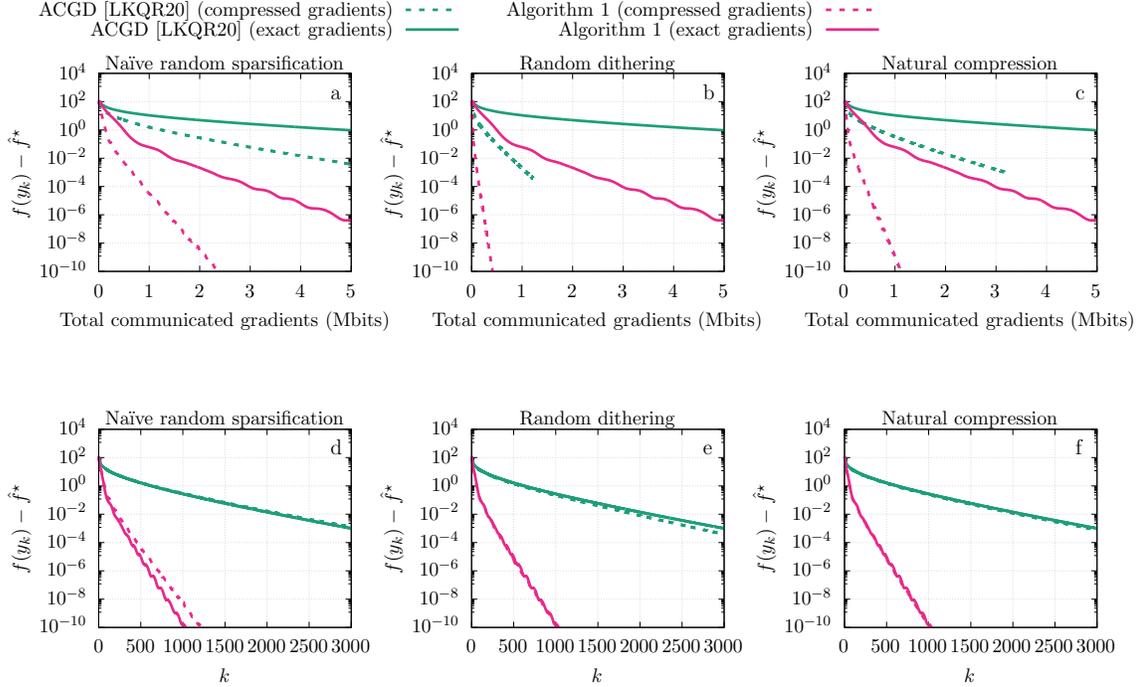}}} 
\caption{Federated logistic regression on the LIBSVM's a5a dataset with $m = 10$ clients.}
\label{fig:logsitic_a5a_appendix}
\end{figure*}

\subsection{Finite-sums}
\label{sec:app_saga}

We now consider the logistic regression problem 
\[
\underset{x \in C}{\text{minimize}} \quad  \sum_{l=1}^m  \log(1 + \exp(-b_l \langle a_l ,  x \rangle))
\] 
where $a_l \in \R^n$ and $b_l \in \{-1,1\}$ are the $l$'th data sample and label respectively. We run Algorithm \ref{al:smooth} with LIBSVM's mushroom dataset with $b = 100$ and parameter $\lambda$ small enough to make the second term in the right-hand-side of Eq.\ \eqref{eq:saga_appendix} nonpositive (computed numerically).  We compare the performance of the algorithm when the gradients are computed exactly (i.e., using all data samples), and when gradients are computed with standard (mini-batch) stochastic gradient descent. The batch size $b$ is fixed to $100$. Figure \ref{fig:saga} shows the simulation results, where we compare the training loss against the number of gradient computations. Observe that Algorithm \ref{al:smooth} outperforms both the mini-batch approach and when the gradient is computed exactly. Furthermore, observe that the convergence rate of the mini-batch gradient descent algorithm slow downs instead of retaining the linear rate of our algorithm.

\begin{figure*}[ht!]
\centering
{\resizebox{0.5\textwidth}{!}{\input{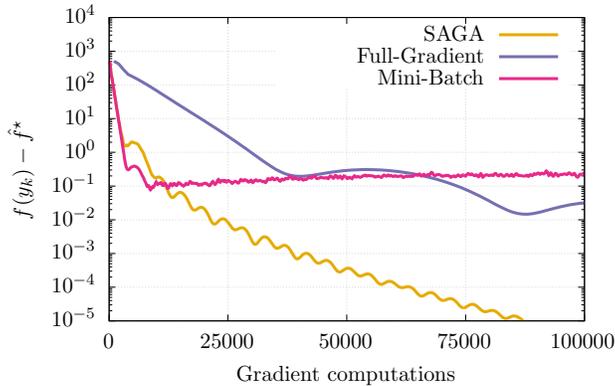}}} 
\caption{Finite-sum problem with the LIBSVM's mushroom dataset. The mini-batch size is equal to $100$.}
\label{fig:saga}
\end{figure*}

\end{document}